\documentclass[11pt,reqno]{amsart}

\usepackage{amsthm}
\usepackage{amscd}
\usepackage{amsfonts}
\usepackage{amssymb}
\usepackage{amsgen}
\usepackage{amsmath}
\usepackage{amsopn}
\usepackage{verbatim}
\usepackage{xypic}
\usepackage{xspace}
\usepackage{multicol}
\usepackage{url}
\usepackage{upref}

\theoremstyle{plain}
\newtheorem{thm}{Theorem}[section]
\newtheorem{lem}[thm]{Lemma}
\newtheorem{prop}[thm]{Proposition}
\newtheorem{cor}[thm]{Corollary}

\theoremstyle{definition}

\theoremstyle{remark}
\newtheorem{rem}[thm]{Remark}

 \font\cyr=wncyr10
 \newcommand{\nc}{\newcommand}

\DeclareMathOperator{\sign}{{sgn}}

\nc{\per}[1]{\underset{#1}{\boldsymbol \pi}\,}

 \nc{\MT}{{\rm MT}}
 \nc{\XX}{{X}}

 \nc{\wht}{{\widehat}}
 \nc{\bwg}{{\bigwedge}}
 \nc{\mmu}{{\boldsymbol{\mu}}}
 \nc{\mal}{{{\scriptstyle \maltese}}}
 \nc{\fA}{{\mathfrak A}}
 \nc{\HH}{{\mathfrak H}}
 \nc{\ra}{\rightarrow}
 \nc{\ors}{{\bfs}}
 \nc{\orr}{{\bfr}}
 \nc{\os}{{\overset}}
 \nc{\G}{{\mathbb G}}
 \nc{\ZZ}{{\mathbb Z}}
 \nc{\RR}{{\mathbb R}}
 \nc{\NN}{{\mathbb N}}
 \nc{\ZN}{{\mathbb Z_{\ge 0}}}
 \nc{\Q}{{\mathbb Q}}
 \nc{\QQ}{{\mathbb Q}}
 \nc{\C}{{\mathbb C}}
 \nc{\Cnn}{{\mathbb C}_{\ge 0}}
 \nc{\Cp}{{\mathbb C}_{>0}}
 \nc{\MPV}{{\mathcal{MPV}}}
 
 \nc{\tB}{{\tilde B}}
 \nc{\Li}{{\rm Li}}
 \nc{\suf}{{\ast\,}}
 \nc{\sufq}{{\ast_q\,}}
 \nc{\gam}{{\gamma}}
 \nc{\gG}{{\Gamma}}
 \nc{\om}{{\omega}}
 \nc{\vep}{{\varepsilon}}
 \nc{\ga}{{\alpha}}
 \nc{\gl}{{\lambda}}
 \nc{\gb}{{\beta}}
 \nc{\gd}{{\delta}}
 \nc{\orgd}{{\vec \gd\,}}
 \nc{\gs}{{\sigma}}
 \nc{\gS}{{\Sigma}}

 \nc{\gk}{{\kappa}}
  \nc{\gz}{{\zeta}}
 \nc{\tgz}{{\tilde{\zeta}}}
 \nc{\gO}{{\Omega}}
 \nc{\sif}{{\mathcal S}}
 \nc{\gt}{{\tau}}
 \nc{\Lra}{\Longrightarrow}
 \nc{\lra}{\longrightarrow}
 \nc{\lmaps}{\longmapsto}
 \nc{\fS}{{\mathfrak S}}
 \nc{\DD}{{\mathfrak D}}
 \nc{\Llra}{\Longleftrightarrow}
 \nc{\ol}{\overline}
 \nc{\ola}{\overleftarrow}
 \nc{\lms}{\longmapsto}
 \nc{\cv}{{{\mathsf c}{\mathsf v}}}
 \nc{\zq}{{\zeta_q}}
 \nc\qup{{q\uparrow 1}}
 \nc{\us}{\underset}
 \nc{\tn}{{\tilde{n}}}
 \nc{\gD}{{\Delta}}
 \nc{\bi}{{\bf i}}
 \nc{\bfone}{{\bf 1}}
 \nc{\bfa}{{\bf a}}
 \nc{\bfb}{{\bf b}}
 \nc{\bfc}{{\bf c}}
 \nc{\bfd}{{\bf d}}
 \nc{\bfe}{{\bf e}}
 \nc{\bff}{{\bf f}}
 \nc{\bfg}{{\bf g}}
 \nc{\bfi}{{\bf i}}
 \nc{\bfj}{{\bf j}}
 \nc{\bfn}{{\bf n}}
 \nc{\bfl}{{\bf l}}
 \nc{\bfk}{{\bf k}}
 \nc{\bfm}{{\bf m}}
 \nc{\bfo}{{\bf o}}
 \nc{\bfp}{{\bf p}}
 \nc{\bfq}{{\bf q}}
 \nc{\bfr}{{\bf r}}
 \nc{\bfs}{{\bf s}}
 \nc{\bft}{{\bf t}}
 \nc{\bfu}{{\bf u}}
 \nc{\bfv}{{\bf v}}
 \nc{\bfw}{{\bf w}}
 \nc{\bfx}{{\bf x}}
 \nc{\bfB}{{\bf B}}
 \nc{\bfP}{{\bf P}}
 \nc{\bfQ}{{\bf Q}}
 \nc{\bfY}{{\bf Y}}
 \nc{\bfgb}{{\boldsymbol \gb}}
 \nc{\bfga}{{\boldsymbol \ga}}
 \nc{\bfrho}{{\boldsymbol \rho}}
 \nc{\bfchi}{{\boldsymbol \chi}}
 \nc{\QX}{{\Q\langle \bfX\rangle}}
 \nc{\QY}{{\Q\langle \bfY\rangle}}
 \nc{\CX}{{\C\langle \bfX\rangle}}
 \nc{\CY}{{\C\langle \bfY\rangle}}
 \nc{\QXX}{{\Q\langle\!\langle \bfX\rangle\!\rangle}}
 \nc{\QYY}{{\Q\langle\!\langle \bfY\rangle\!\rangle}}
 \nc{\CXX}{{\C\langle\!\langle \bfX\rangle\!\rangle}}
 \nc{\CYY}{{\C\langle\!\langle \bfY\rangle\!\rangle}}

 \nc{\bbA}{{\mathbb A}}
 \nc{\bbB}{{\mathbb B}}
 \nc{\bbC}{{\mathbb C}}
 \nc{\bbD}{{\mathbb D}}
 \nc{\bbE}{{\mathbb E}}
 \nc{\bbF}{{\mathbb F}}
 \nc{\bbG}{{\mathbb G}}
 \nc{\bbH}{{\mathbb H}}
 \nc{\bbI}{{\mathbb I}}
 \nc{\bbJ}{{\mathbb J}}
 \nc{\bbK}{{\mathbb K}}
 \nc{\bbL}{{\mathbb L}}
 \nc{\bbM}{{\mathbb M}}
 \nc{\bbN}{{\mathbb N}}
 \nc{\bbO}{{\mathbb O}}
 \nc{\bbP}{{\mathbb P}}
 \nc{\bbQ}{{\mathbb Q}}
 \nc{\bbR}{{\mathbb R}}
 \nc{\bbS}{{\mathbb S}}
 \nc{\bbT}{{\mathbb T}}
 \nc{\bbU}{{\mathbb U}}
 \nc{\bbV}{{\mathbb V}}
 \nc{\bbW}{{\mathbb W}}
 \nc{\bbX}{{\mathbb X}}
 \nc{\bbY}{{\mathbb Y}}
 \nc{\bbZ}{{\mathbb Z}}
 \nc{\bba}{{\mathbb a}}
 \nc{\bbb}{{\mathbb b}}
 \nc{\bbc}{{\mathbb c}}
 \nc{\bbd}{{\mathbb d}}
 \nc{\bbe}{{\mathbb e}}
 \nc{\bbf}{{\mathbb f}}
 \nc{\bbg}{{\mathbb g}}
 \nc{\bbh}{{\mathbb h}}
 \nc{\bbi}{{\mathbb i}}
 \nc{\bbk}{{\mathbb k}}
 \nc{\bbl}{{\mathbb l}}
 \nc{\bbm}{{\mathbb m}}
 \nc{\bbn}{{\mathbb n}}
 \nc{\bbo}{{\mathbb o}}
 \nc{\bbp}{{\mathbb p}}
 \nc{\bbq}{{\mathbb q}}
 \nc{\bbr}{{\mathbb r}}
 \nc{\bbs}{{\mathbb s}}
 \nc{\bbt}{{\mathbb t}}
 \nc{\bbu}{{\mathbb u}}
 \nc{\bbv}{{\mathbb v}}
 \nc{\bbw}{{\mathbb w}}
 \nc{\bbx}{{\mathbb x}}
 \nc{\bby}{{\mathbb y}}
 \nc{\bbz}{{\mathbb z}}

 \nc{\calA}{{\mathcal A}}
 \nc{\calB}{{\mathcal B}}
 \nc{\calC}{{\mathcal C}}
 \nc{\calD}{{\mathcal D}}
 \nc{\calE}{{\mathcal E}}
 \nc{\calF}{{\mathcal F}}
 \nc{\calG}{{\mathcal G}}
 \nc{\calH}{{\mathcal H}}
 \nc{\calI}{{\mathcal I}}
 \nc{\calJ}{{\mathcal J}}
 \nc{\calK}{{\mathcal K}}
 \nc{\calL}{{\mathcal L}}
 \nc{\calM}{{\mathcal M}}
 \nc{\calN}{{\mathcal N}}
 \nc{\calO}{{\mathcal O}}
 \nc{\calP}{{\mathcal P}}
 \nc{\calQ}{{\mathcal Q}}
 \nc{\calR}{{\mathcal R}}
 \nc{\calS}{{\mathcal S}}
 \nc{\calT}{{\mathcal T}}
 \nc{\calU}{{\mathcal U}}
 \nc{\calV}{{\mathcal V}}
 \nc{\calW}{{\mathcal W}}
 \nc{\calX}{{\mathcal X}}
 \nc{\calY}{{\mathcal Y}}
 \nc{\calZ}{{\mathcal Z}}
  \nc{\cala}{{\mathcal a}}
 \nc{\calb}{{\mathcal b}}
 \nc{\calc}{{\mathcal c}}
 \nc{\cald}{{\mathcal d}}
 \nc{\cale}{{\mathcal e}}
 \nc{\calf}{{\mathcal f}}
 \nc{\calg}{{\mathcal g}}
 \nc{\calh}{{\mathcal h}}
 \nc{\cali}{{\mathcal i}}
 \nc{\calj}{{\mathcal j}}
 \nc{\calk}{{\mathcal k}}
 \nc{\call}{{\mathcal l}}
 \nc{\calm}{{\mathcal m}}
 \nc{\caln}{{\mathcal n}}
 \nc{\calo}{{\mathcal o}}
 \nc{\calp}{{\mathsf p}}
 \nc{\calq}{{\mathcal q}}
 \nc{\calr}{{\mathcal r}}
 \nc{\cals}{{\mathcal s}}
 \nc{\calt}{{\mathcal t}}
 \nc{\calu}{{\mathcal u}}
 \nc{\calv}{{\mathcal v}}
 \nc{\calw}{{\mathcal w}}
 \nc{\calx}{{\mathcal x}}
 \nc{\caly}{{\mathcal y}}
 \nc{\calz}{{\mathcal z}}

 \nc{\frakA}{{\mathfrak A}}
 \nc{\frakB}{{\mathfrak B}}
 \nc{\frakC}{{\mathfrak C}}
 \nc{\frakD}{{\mathfrak D}}
 \nc{\frakE}{{\mathfrak E}}
 \nc{\frakF}{{\mathfrak F}}
 \nc{\frakG}{{\mathfrak G}}
 \nc{\frakH}{{\mathfrak H}}
 \nc{\frakI}{{\mathfrak I}}
 \nc{\frakJ}{{\mathfrak J}}
 \nc{\frakK}{{\mathfrak K}}
 \nc{\frakL}{{\mathfrak L}}
 \nc{\frakM}{{\mathfrak M}}
 \nc{\frakN}{{\mathfrak N}}
 \nc{\frakO}{{\mathfrak O}}
 \nc{\frakP}{{\mathfrak P}}
 \nc{\frakQ}{{\mathfrak Q}}
 \nc{\frakR}{{\mathfrak R}}
 \nc{\frakS}{{\mathfrak S}}
 \nc{\frakT}{{\mathfrak T}}
 \nc{\frakU}{{\mathfrak U}}
 \nc{\frakV}{{\mathfrak V}}
 \nc{\frakW}{{\mathfrak W}}
 \nc{\frakX}{{\mathfrak X}}
 \nc{\frakY}{{\mathfrak Y}}
 \nc{\frakZ}{{\mathfrak Z}}
 \nc{\fraka}{{\mathfrak a}}
 \nc{\frakb}{{\mathfrak b}}
 \nc{\frakc}{{\mathfrak c}}
 \nc{\frakd}{{\mathfrak d}}
 \nc{\frake}{{\mathfrak e}}
 \nc{\frakf}{{\mathfrak f}}
 \nc{\frakg}{{\mathfrak g}}
 \nc{\frakh}{{\mathfrak h}}
 \nc{\fraki}{{\mathfrak i}}
 \nc{\frakj}{{\mathfrak j}}
 \nc{\frakk}{{\mathfrak k}}
 \nc{\frakl}{{\mathfrak l}}
 \nc{\frakm}{{\mathfrak m}}
 \nc{\frakn}{{\mathfrak n}}
 \nc{\frako}{{\mathfrak o}}
 \nc{\frakp}{{\mathfrak p}}
 \nc{\frakq}{{\mathfrak q}}
 \nc{\frakr}{{\mathfrak r}}
 \nc{\fraks}{{\mathfrak s}}
 \nc{\frakt}{{\mathfrak t}}
 \nc{\fraku}{{\mathfrak u}}
 \nc{\frakv}{{\mathfrak v}}
 \nc{\frakw}{{\mathfrak w}}
 \nc{\frakx}{{\mathfrak x}}
 \nc{\fraky}{{\mathfrak y}}
 \nc{\frakz}{{\mathfrak z}}
 \nc{\sha}{{\mbox{\cyr x}}}

\begin{document}

\title[Congruences of alternating multiple harmonic sums]
{Congruences of alternating multiple harmonic sums}

\author{Roberto Tauraso$^\dag$}
\author{Jianqiang Zhao$^\ddag$}

\maketitle

\begin{center}
{}$^\dag$Dipartimento di Matematica,
Universit\`a di Roma ``Tor Vergata'', Italy
\end{center}

\begin{center}
{}$^\ddag$Department of Mathematics, Eckerd College, St. Petersburg, FL 33711, USA\\
{}$^\ddag$Max-Planck Institut f\"ur Mathematik, Vivatsgasse 7, 53111 Bonn, Germany
\end{center}

\medskip
\noindent{\small {\bf Abstract.}
In this sequel to \cite{Tau1}, we continue to study the congruence
properties of the alternating version of multiple harmonic sums. As contrast to
the study of multiple harmonic sums where Bernoulli numbers and Bernoulli polynomials
play the key roles, in the alternating setting the Euler numbers
and the Euler polynomials are also essential.



\vskip0.7cm
\section{Introduction}
In proving a Van Hamme type congruence the first author was led to
consider some congruences involving alternating multiple harmonic sums
(AMHS for short) which are defined as follows.
Let $d>0$ and let $\ors:=(s_1,\dots, s_d)\in (\ZZ^*)^d$.
We define the \emph{alternating multiple harmonic sum} as
\begin{equation*}
 H(\ors;n):=\sum_{1\leq k_1<k_2<\dots<k_d\leq n}\;
\prod_{i=1}^d \frac{\sign(s_i)^{k_i}}{k_i^{|s_i|}}.
\end{equation*}
By convention we set $H(\ors;n)=0$ any $n<d$. We call
$\ell(\ors):=d$ and $|\ors|:=\sum_{i=1}^d |s_i|$ its
\emph{depth} and \emph{weight}, respectively. We point out
that $\ell(\ors)$ is sometimes called length in the literature.
When every $s_i$ is positive we recover the multiple harmonic
sums (MHS for short) whose congruence properties are
studied in \cite{Hmodp,Haspect,1stpart,2ndpart}.
There is another ``non-strict''
version of the AMHS defined as follows:
\begin{equation*}
 S(\ors;n):=\sum_{1\leq k_1\leq k_2\leq\dots\leq k_d\leq n}\;
\prod_{i=1}^d{\sign(s_i)^{k_i}\over k_i^{|s_i|}}.
\end{equation*}
By Inclusion and Exclusion Principle it is easy to see that
\begin{align} \label{equ:StoH}
S(\ors;n)=&\sum_{\orr\preceq \ors} H(\orr;n), \\
H(\ors;n)=&\sum_{\orr\preceq \ors} (-1)^{\ell(\ors)-\ell(\orr)}S(\orr;n),
\label{equ:HtoS}
\end{align}
where $\orr\prec \ors$ means $\orr$ can be obtained from $\ors$ by
combining some of its parts.

The main goal of this paper is to provide a systematic study of
the congruence property of $H(\ors;p-1)$ (and $S(\ors;p-1)$) for primes $p>|\ors|+2$
by using intimate relations between Bernoulli polynomials, Bernoulli
numbers, Euler polynomials, and Euler numbers. Throughout the paper, we often
use the abbreviation $S(-)=S(-;p-1)$ and $H(-)=H(-;p-1)$ if no confusion will arise.
The following congruences concerning harmonic sums will be crucial for us.
\begin{thm} \label{thm:sun}
 Let $k\in \NN$ and $p$ be a prime. Set
$\XX_p(k):=\frac{B_{p-k}}{p-k}-\frac{B_{2p-1-k}}{2(2p-1-k)}.$

(a) \emph{\cite[Theorem 5.1]{Sun}} If $p\ge k+3$ then
\begin{equation}\label{equ:SunThm5.1}
H(k)\equiv \left\{
             \begin{array}{ll}
         k(k+1)B_{p-2-k}p^2/2(p-2-k)  &\pmod{p^3}, \quad \hbox{ if $k$ is odd;} \\
          -2k \XX_p(k+1) p  &\pmod{p^3}, \quad \hbox{ if $k$ is even.}
             \end{array}
           \right.
\end{equation}

(b) \emph{\cite[Theorem 5.2]{Sun}} If $p\ge k+4$ then
\begin{equation}\label{equ:SunThm5.2}
H(k;(p-1)/2)\equiv \left\{
             \begin{array}{ll}
2(2^k-2)\XX_p(k)
    &\pmod{p^2},\quad  \hbox{ if $k>1$ is odd;} \phantom{\frac12} \\
 -k (2^{k+1}-1)\XX_p(k+1)  p
    & \pmod{p^3},\quad \hbox{ if $k$ is even;} \phantom{\frac12} \\
-2q_p+pq_p^2-\frac23p^2q_p^3- \frac7{12}p^2 B_{p-3}
    &\pmod{p^3},\quad  \hbox{ if $k=1$.}
             \end{array}
           \right.
\end{equation}
Here $q_p=(2^{p-1}-1)/p$ is the Fermat quotient.
\end{thm}

We now sketch the outline of the paper.
We start \S\ref{sec:properties} by recalling some important
relations among AMHS such as the stuffle and reversal relations.
Then we present some basic properties of Euler polynomials which
provide one of the fundamental tools for us in the alternating setting.
Then in \S\ref{sec:red} we describe two reduction procedures for $H(\ors)\pmod{p}$
general $\ors$, which are used to derive congruences in depth two and depth
three cases in \S\ref{sec:depth2} and \S\ref{sec:depth3}, respectively.
\begin{thm} \label{thm:reductionMain}
Let $a,\ell\in \NN$, $\ors=(s_1,\dots,s_\ell)\in (\ZZ^*)^\ell$
and $\ors'=(s_2,\dots,s_\ell)$. For every prime $p\ge a+2$ write $H(-)=H(-;p-1)$.
Then we have
the reduction formulae
\begin{align*}
   H(a,\ors)\equiv&  -\frac{1}{a} H\big((a-1)\oplus s_1, \ors'\big)-\frac12 H(a\oplus s_1, \ors')\\
 &+\sum_{k=2}^{p-1-a}  {p-a\choose k} \frac{B_k}{p-a} H\big((k+a-1)\oplus s_1, \ors'\big)\pmod{p},\\
   H(-a,\ors)\equiv& \frac{(1-2^{p-a})B_{p-a}}{p-a}  \Big(H(\ors)-H(-s_1, \ors')\Big)\\
  & -\sum_{k=0}^{p-2-a}{p-1-a\choose k}\frac{E_k(0)}{2}
    H\big((k+a)\oplus(-s_1), \ors'\big)\pmod{p},
\end{align*}
where $s\oplus t=\sign(st)(|s|+|t|)$ and $E_k(0)= 2(1-2^{k+1})B_{k+1}/(k+1).$
\end{thm}
In \S\ref{sec:depthAll} we deal with the homogeneous AMHS of arbitrary
depth and provide an explicit formula using the relation between the power
sum and elementary symmetric functions and the partition functions.
\S\ref{sec:wt4} is devoted to a comprehensive study of the weight four
AMHS in which identities involving Bernoulli numbers such as those
proved in \cite{1stpart} play the leading roles. For example,
by writing $H(-)=H(-;p-1)$ we find
the following interesting relations (see Proposition~\ref{prop:wt4depth2},
Proposition~\ref{prop:wt4depth3} and Proposition~\ref{prop:wt4depth3Eg}):
\begin{alignat*}{2}
 H(1,-3)\equiv  \frac12 H(-2,2) \equiv&
    \sum_{k=0}^{p-3} 2^k B_kB_{p-3-k} &\pmod{p}, \\
     H(-1,3)\equiv&  -\frac12 q_pB_{p-3} &\pmod{p},\\
H(1,-2,-1)\equiv&   H(1,-3) -\frac54 q_pB_{p-3} &\pmod{p}, \\
H(-1,1,-1,1)\equiv H(1,-1,1,-1)\equiv& -\frac1{12}
    \Big( q_pB_{p-3} + 2q_p^4\Big) &\pmod{p},\\
H(-1,1,1,-1)\equiv&  \frac1{12} \Big(6H(1,-3)+7q_pB_{p-3}+2q_p^4\Big)
    &\pmod{p},
\end{alignat*}
for all primes $p\ge 7$.
None of the above congruences can be obtained simply by the stuffle and
reversal relations.

After studying some special types of AMHS of weight four in \S\ref{sec:depthAll},
we turn to congruence relations involving lower weight AMHS
modulo higher powers of primes in the last section. One of the main ideas
in these sections is to relate AMHS to sums $U(\ors;n)$ and $V(\ors;n)$
defined by \eqref{equ:Udefn} and \eqref{equ:Vdefn}, respectively. These sums
appeared previously in congruences involving powers of Fermat quotient
(see \cite{Agoh,Bach,DS,Gran}).

Most of results of this paper were obtained while the second author was
visiting the Max-Planck-Institut f\"ur Mathematik whose support is
gratefully acknowledged.

\section{Properties of AMHS}\label{sec:properties}
\subsection{Stuffle relation.} \label{sec:stuffle}
The most important relation between
AMHS is the so called stuffle relation. It is possible to formalize
this using words as in \cite[\S2]{Zesum} or \cite[\S 2.2]{Rac}
which is a generalization of the MHS case (see \cite[\S2]{Hmodp}).
Unfortunately, for AMHS we don't have the integral representations
which provide another product structure for the alternating multiple
zeta values which are the infinite sum version of AMHS.

Fix a positive integer $n$. Let $\fA$ be the algebra generated
by letters $y_s$ for $s\in \ZZ^*$. Define a multiplication $\ast$
on $\fA$ by requiring that $\ast$ distribute over addition,
that $\bfone\ast w=w\ast \bfone=w$ for the empty word $\bfone$
and any word $w$, and that, for any two words $w_1,w_2$ and
two letters $y_s,y_t$ ($s,t\in \ZZ^*$)
\begin{equation}\label{equ:defnstuffle}
y_sw_1\ast y_t w_2 = y_s(w_1\ast y_t w_2) + y_t(y_s w_1\ast w_2)+
  y_{s\oplus t}(w_1\ast w_2)
\end{equation}
where $s\oplus t=\sign(st)(|s|+|t|)$. Then we get an algebra
homomorphism
\begin{align*}
 H:\quad (\fA,*) \quad\lra & \ \{H(\ors;n):\ors\in \ZZ^r,r\in \NN\}\\
 \qquad \bfone \qquad \lms &\qquad  1\\
 y_{s_1}\dots y_{s_r}\lms &\quad  H(s_1,\dots, s_r;n).
\end{align*}
For example,
\begin{multline*}
 H(-2;n)H(-3,2;n)=H(-2,-3,2;n)+H(-3,-2,2;n)+H(-3,2,-2;n)\\
 +H(5,2;n)+H(-3,-4;n).
\end{multline*}

There is another kind of relation caused by the reversal of the
arguments which we call the \emph{reversal relations}. For any
$\ors=(s_1,\dots,s_r)$ they have the form
\begin{equation}\label{equ:reversal}
\aligned
H(\ors;p-1)\equiv& \sign\Big(\prod_{j=1}^r s_j\Big)(-1)^r H(\ola{\ors};p-1) \pmod{p},\\
S(\ors;p-1)\equiv& \sign\Big(\prod_{j=1}^r s_j\Big)(-1)^r S(\ola{\ors};p-1) \pmod{p},
\endaligned
\end{equation}
for any odd prime $p>|\ors|$, where $\ola{\ors}=(s_r,\dots,s_1)$.

\subsection{Euler polynomials.}
In the study of congruences of MHS \cite{Hmodp,1stpart,2ndpart} we have
seen that Bernoulli numbers play the key roles by virtue of the following
identity: (\cite[p.~804, 23.1.4-7]{AS})
\begin{equation}\label{equ:sump}
\sum_{j=1}^{n-1} j^d=\sum_{r=0}^d {d+1\choose r}\frac{B_r}{d+1}
n^{d+1-r}, \quad \forall n,d\ge 1.
\end{equation}
In the case of AMHS, however,
the Euler polynomials and the Euler numbers are indispensable, too.
Recall that the Euler polynomials $E_n(x)$ are
defined by the generating function
\begin{equation*}
\frac{2 e^{tx} }{e^t+1}
=\sum_{n=0}^\infty E_n(x) \frac{t^n}{n!}.
\end{equation*}

\begin{lem} \label{lem:altsum}
Let $n\in \ZN$. Then we have
\begin{equation}\label{equ:altsumEuler}
 \sum_{i=1}^{d-1} (-1)^i  i^n= \frac 12\Big((-1)^{d-1}E_n(d)+ E_{n}(0)\Big)
=\sum_{a=0}^n {n \choose a} F_{n,d,a}  d^{n-a},
\end{equation}
where
\begin{equation*}
F_{n,d,a}=\left\{
            \begin{array}{ll}
               (-1)^{d-1} E_a(0)/2, & \hbox{if $a<n$;} \\
              (1-(-1)^d)E_n(0)/2, & \hbox{if $a=n>0$;} \\
              -(1+(-1)^d)/2, & \hbox{if $a=n=0$.}
            \end{array}
          \right.
\end{equation*}
Moreover, $E_0(0)=1$ and for all $a\in \NN$
\begin{equation}\label{equ:EulerBern}
E_a(0) = \frac{2^{a+1}}{a+1} \Big(B_{a+1}\Big(\frac{1}2\Big)
-B_{a+1}\Big)= \frac{2}{a+1}(1-2^{a+1})B_{a+1}.
\end{equation}
\end{lem}
\begin{proof}
Consider the generating function
 \begin{align}
  \sum_{n=0}^\infty\left( \sum_{i=1}^{d-1} (-1)^i  i^n \right) \frac{t^n}{n!}
=&  \sum_{i=1}^{d-1} (-1)^i e^{ti}  \notag  \\
=& \frac{  (-e^t)^d-1}{-e^t-1}-1    \notag \\
=& \frac{(-1)^{d-1} e^{dt}+1}{e^t+1}-1    \notag \\
=&\frac 12 \sum_{n=0}^\infty \Big((-1)^{d-1} E_n(d)+ E_{n}(0)\Big) \frac{t^n}{n!}-1.
\label{equ:genequ}
\end{align}
Now \eqref{equ:altsumEuler} follows from the notorious equation
(see for e.g., \cite[p.~805, 23.1.7]{AS})
\begin{equation}\label{equ:altsumEulerpoly}
E_n(x)=\sum_{a=0}^n {n \choose a} E_a(0) x^{n-a}
\end{equation}
for all $n>0$. Equation \eqref{equ:EulerBern} is also well-known
(see for e.g., item 23.1.20 on p.~805 of loc.\ cit.).
\end{proof}

\begin{rem}
The classical Euler numbers $E_k$ is defined by
$$\frac{2}{e^t+ e^{-t}} = \sum_{k=0}^{\infty} E_k  \frac{t^k}{k!} .$$
They are related to $E_k(0)$ by the formula (see \cite[p.~805, 23.1.7]{AS})
$$E_m(0)=\sum_{k=0}^m {m \choose k} \frac{E_k}{2^k}
    \left(-\frac{1}{2}\right)^{m-k}.$$
\end{rem}

\begin{cor} \label{cor:Hdepth1}
Let $a\in \ZN$ and $p$ be a prime such that $p\ge a+2$. Then
\begin{equation}\label{equ:Hdepth1}
H(-a;p-1)\equiv
    \left\{
      \begin{array}{ll}
         {\displaystyle -\frac{2(1-2^{p-a})}{a}B_{p-a} \text{\raisebox{-10pt}{\,} }}& \pmod{p},  \ \quad \hbox{if $a$ is odd;} \\
         {\displaystyle \frac{a(1-2^{p-1-a})}{a+1}p B_{p-1-a}} & \pmod{p^2}, \quad \hbox{if $a$ is even.}
      \end{array}
    \right.
\end{equation}
\end{cor}
\begin{proof}
Taking $d=p$ and $n=p(p-1)-a$ in the Lemma we see that
\begin{align*}
H(-a;p-1)\equiv &  F_{p(p-1)-a,p,p(p-1)}+p(p(p-1)-a) F_{p(p-1)-a,p,p(p-1)-1-a} \\
\equiv & E_{p(p-1)-a}(0)-\frac12 p a E_{p(p-1)-1-a}(0) \pmod{p^2},
\end{align*}
since all the coefficients in \eqref{equ:altsumEuler} are $p$-integral
by \eqref{equ:EulerBern} and the property of Bernoulli numbers:
$B_m$ is not $p$-integral if and only if $p-1$ divides $m>0$.
Then the corollary directly follows from \eqref{equ:EulerBern}
and Kummer congruences
\begin{align*}
\frac{B_{p(p-1)-a}}{p(p-1)-a}\equiv &\frac{B_{p-1-a}}{p-1-a} & \hskip-3cm \pmod{p},\\
\frac{B_{p(p-1)-a+1}}{p(p-1)-a+1}\equiv& \ \frac{B_{p-a}}{p-a} &\hskip-3cm \pmod{p}.
\end{align*}
\end{proof}
\begin{rem} (a). The corollary can also be obtained from
\cite[Theorem~2.1]{Tau1} combined with \eqref{equ:SunThm5.2}. Notice that
both terms in \cite[Theorem~2.1]{Tau1} contribute nontrivially when
$a$ is even since the modulus is $p^2$. (b). Notice that if $a$ is odd we allow
$p=a+2$ to be a prime modulus unlike \cite[Lemma 5.1]{CD}. If $a$ is even
then the corollary is given also by \cite[(6.2)]{CD}.
\end{rem}

\subsection{Two reduction formulae}\label{sec:red}
We now prove two reduction formulae of $H(\ors)$ for
arbitrary composition $\ors$, corresponding to the two
cases where $\ors$ begins with a positive or a negative number.
\begin{thm} \label{thm:reduction}
Let $a,\ell\in \NN$, $\ors=(s_1,\dots,s_\ell)\in (\ZZ^*)^\ell$
and $\ors'=(s_2,\dots,s_\ell)$. For any prime $p\ge a+2$ write $H(-)=H(-;p-1)$. Then
\begin{equation}\label{equ:reduction}
   H(a,\ors)\equiv  -\frac{1}{a} H\big((a-1)\oplus s_1, \ors'\big)
   +\sum_{k=1}^{p-1-a}  {p-a\choose k} \frac{B_k}{p-a} H\big((k+a-1)\oplus s_1, \ors'\big)  \pmod{p}.
\end{equation}
\end{thm}
\begin{proof}
By definition
\begin{align*}
H(a,\ors) \equiv&\sum_{j_\ell=1}^{p-1} \frac{\sign(s_\ell)^{j_\ell}}{j_\ell^{|s_\ell|} }
    \sum_{j_{\ell-1}=1}^{j_\ell-1}
    \frac{\sign(s_{\ell-1})^{j_{\ell-1}}}{j_{\ell-1}^{|s_{\ell-1}|}}
    \cdots \sum_{j_1=1}^{j_2-1}\frac{\sign(s_1)^{j_1}}{j_1^{|s_1|}}\sum_{j=1}^{j_1-1} j^{p-1-a} \\
\equiv&\sum_{k=0}^{p-1-a}  {p-a\choose k} \frac{B_k}{p-a} \sum_{j_\ell=1}^{p-1}
    \frac{\sign(s_\ell)^{j_\ell}}{j_\ell^{|s_\ell|} } \sum_{j_{\ell-1}=1}^{j_\ell-1}
    \frac{\sign(s_{\ell-1})^{j_{\ell-1}}}{j_{\ell-1}^{|s_{\ell-1}|}}
    \cdots \sum_{j_1=1}^{j_2-1}\frac{\sign(s_1)^{j_1}}{j_1^{|s_1|}} j_1^{p-a-k}
\end{align*}
which is exactly the right hand side of \eqref{equ:reduction}.
\end{proof}

\begin{thm} \label{thm:reduction2}
Let $a,\ell\in \NN$, $\ors=(s_1,\dots,s_\ell)\in (\ZZ^*)^\ell$
and $\ors'=(s_2,\dots,s_\ell)$. For any prime $p\ge a+2$ write $H(-)=H(-;p-1)$. Then
\begin{align*}
   H(-a,\ors)\equiv& \frac{(1-2^{p-a})B_{p-a}}{p-a}  \Big(H(\ors)-H(-s_1, \ors')\Big)\\
   -&  \sum_{k=0}^{p-2-a}{p-1-a\choose k} \frac{(1-2^{k+1})B_{k+1}}{k+1}
    H\big((k+a)\oplus(-s_1), \ors'\big)  \pmod{p}.
\end{align*}
\end{thm}
\begin{proof}
By definition and Lemma \ref{lem:altsum}
\begin{align*}
H(a,\ors) \equiv&\sum_{j_\ell=1}^{p-1} \frac{\sign(s_\ell)^{j_\ell}}{j_\ell^{|s_\ell|}}
    \sum_{j_{\ell-1}=1}^{j_\ell-1}\frac{\sign(s_{\ell-1})^{j_{\ell-1}}}{j_{\ell-1}^{|s_{\ell-1}|}}
    \cdots \sum_{j_1=1}^{j_2-1}   \frac{\sign(s_1)^{j_1}}{j_1^{|s_1|} }
    \sum_{j=1}^{j_1-1}(-1)^j j^{p-1-a} \\
\equiv&\sum_{k=0}^{p-2-a}  {p-1-a\choose k} \frac{E_k(0)}{2}\sum_{j_\ell=1}^{p-1}
    \frac{\sign(s_\ell)^{j_\ell}}{j_\ell^{|s_\ell|}} \sum_{j_{\ell-1}=1}^{j_\ell-1}
    \frac{\sign(s_{\ell-1})^{j_{\ell-1}}}{j_{\ell-1}^{|s_{\ell-1}|}  }
    \cdots \sum_{j_1=1}^{j_2-1}   \frac{\sign(s_1)^{j_1}}{j_1^{|s_1|+k+a} }
    (-1)^{j_1-1} \\
&+\frac{E_{p-1-a}(0)}{2}
    \sum_{j_\ell=1}^{p-1} \frac{\sign(s_\ell)^{j_\ell}}{j_\ell^{|s_\ell|} }
    \sum_{j_{\ell-1}=1}^{j_\ell-1}
    \frac{\sign(s_{\ell-1})^{j_{\ell-1}}}{j_{\ell-1}^{|s_{\ell-1}|}}
    \cdots \sum_{j_1=1}^{j_2-1}\frac{\sign(s_1)^{j_1}}{j_1^{|s_1|}} \Big(1-(-1)^{j_1} \Big)\\
\equiv& \frac{E_{p-1-a}(0)}{2} \Big(H(\ors)-H(-s_1, \ors')\Big)
    -\sum_{k=0}^{p-2-a}  {p-1-a\choose k} \frac{E_k(0)}{2} H\big((k+a)\oplus(-s_1), \ors'\big).
\end{align*}
The theorem follows from \eqref{equ:EulerBern} easily.
\end{proof}

\section{AMHS of depth two}\label{sec:depth2}
In this section we will provide congruence formulae for all depth two AMHS.
All but one case are given by very concise values involving
Bernoulli numbers or Euler numbers (which are closely related by
the identity \eqref{equ:EulerBern}).
\begin{thm} \label{thm:depth2}
Let $a,b\in \NN$ and a prime $p\ge a+b+2$. Write $S(-)=S(-;p-1)$
and $H(-)=H(-;p-1)$. If $a+b$ is odd then we have
\begin{alignat}{2}
S(a,b)\equiv H(a,b)\equiv & \frac{(-1)^b }{a+b}{a+b\choose a} B_{p-a-b}&\pmod{p},  \label{equ:l2wtodd1} \\
H(-a,-b)\equiv&  \frac{2^{p-a-b}-1}{a+b}(-1)^b {a+b\choose a} B_{p-a-b} &\pmod{p}, \label{equ:l2wtodd2}  \\
S(-a,-b)\equiv&  \frac{2^{p-a-b}-1}{a+b} \left(2+(-1)^b{a+b\choose a} \right)B_{p-a-b} &\pmod{p}, \label{equ:l2wtodd2S}  \\
H(-a,b)\equiv H(a,-b)\equiv  &
\frac{1-2^{p-a-b}}{a+b}B_{p-a-b} &\pmod{p}, \label{equ:l2wtodd}\\
S(-a,b)\equiv S(a,-b)\equiv  &
\frac{2^{p-a-b}-1}{a+b}B_{p-a-b} &\pmod{p}. \label{equ:l2wtoddS}
\end{alignat}
If $a+b$ is even then we have
\begin{alignat}{2}\label{equ:l2wteven1}
S(a,b)\equiv H(a,b)\equiv& 0 &\pmod{p}, \\
 \label{equ:l2wteven2}
S(-a,-b)\equiv H(-a,-b)\equiv&  \frac{2(1-2^{p-a})(1-2^{p-a})}{ab}B_{p-a}B_{p-b} \qquad &\pmod{p}.
\end{alignat}
\end{thm}

\begin{proof} Congruences \eqref{equ:l2wtodd1} and \eqref{equ:l2wteven1} follow from
\cite[Theorem~3.1]{1stpart} (see \cite[Theorem~6.1]{Hmodp} for a different proof).
Congruences \eqref{equ:l2wtodd2}, \eqref{equ:l2wtodd}, and \eqref{equ:l2wteven2}
are given by \cite[Lemma 6.2]{CD} (notice that $(1-2^{p-1})B_{p-1}\equiv q_p \pmod{p}$).
Finally, congruences for $S$ version of AMHS
are all easy consequence of the $H$ version by the relation
$S(\ga,\gb)=H(\ga,\gb)+H(\ga\oplus\gb)$.
\end{proof}

Even though we don't have compact congruence formulae for $H(-a,b)$
and $H(a,-b)$ when $a+b$ is even we can prove two
general statements using the two reduction procedures
provided by Theorem~\ref{thm:reduction} and Theorem~\ref{thm:reduction2}.

\begin{prop}  \label{prop:wteven}
Let $a,b\in \NN$ and a prime $p\ge a+b+2$. Write $S(-)=S(-;p-1)$
and $H(-)=H(-;p-1)$. If $a+b$ is even then we have
\begin{multline}
 \label{equ:wteven}
S(a,-b)\equiv H(a,-b)
\equiv -S(-b,a)\equiv -H(-b,a)  \\
\equiv \sum_{k=0}^{p-a-1}{p-a\choose k}
\frac{2(1-2^{2p-a-b-k})B_kB_{2p-a-b-k}}{(p-a)(2p-a-b-k)} \pmod{p}.
\end{multline}
\end{prop}
\begin{proof}  By Theorem~\ref{thm:reduction} we have
\begin{equation*}
H(a,-b)\equiv
\sum_{k=0}^{p-a-1} {p-a \choose k}  \frac{B_k}{p-a} \cdot H(-(a+b+k-1))\pmod{p}.
\end{equation*}
To use Cororllary~\ref{cor:Hdepth1} we need to break the sum into two parts, i.e., when
$a+b+k<p$ and when $a+b+k\ge p$. In the first case we can replace $k$ by $k+p-1$
and then to get the correct term in \eqref{equ:wteven} we only need to
use Fermat's Little Theorem $2^{p+1-a-b-k}\equiv 2^{2p-a-b-k}\pmod{p}$
and Kummer congruence
$B_{p+1-a-b-k}/(p+1-a-b-k)\equiv  B_{2p-a-b-k}/(2p-a-b-k)\pmod{p}$.
This finishes the proof of the proposition.
\end{proof}

The second reduction procedure, Theorem \ref{thm:reduction2}, provides us
another useful result on AMHS of even weight and depth two.
\begin{prop}  \label{prop:wteven2}
Let $a,b\in \NN$ and a prime $p\ge a+b+2$. Write $S(-)=S(-;p-1)$
and $H(-)=H(-;p-1)$. If $a+b$ is even then we have
\begin{multline*}
S(-a,b)\equiv H(-a,b)\equiv -S(b,-a)\equiv -H(b,-a) \\
\equiv
 \sum_{k=1}^{p-2-a-b} {p-1-a\choose k}
    \frac{ 2 (1-2^{k+1})(1-2^{p-a-b-k})B_{k+1} B_{p-a-b-k}}{(k+1)(a+b+k)} \\
 +\sum_{k=p-1-a-b}^{p-1-a} {p-1-a\choose k}
     \frac{2 (1-2^{k+1})(1-2^{2p-1-a-b-k})B_{k+1} B_{2p-1-a-b-k}}{(k+1)(1+a+b+k)}\pmod{p}.
\end{multline*}
\end{prop}
\begin{proof}
By Theorem \ref{thm:reduction2} we have
\begin{equation*}
   H(-a,b)\equiv \sum_{k=0}^{p-1-a}{p-1-a\choose k} \frac{(1-2^{k+1})B_{k+1}}{k+1}
    H\big(-(k+a+b)\big) \pmod{p}
\end{equation*}
since $H(b)\equiv 0\pmod{p}$. The rest follows from Cororllary~\ref{cor:Hdepth1} similar
to the proof of Proposition~\ref{prop:wteven}.
\end{proof}

The two propositions above will be used in \S\ref{sec:wt4}
to compute some AMHS congruences explicitly.

\section{AMHS of depth three}\label{sec:depth3}
All the congruences in this section are modulo a prime $p$.
Write $H(-)=H(-;p-1)$.
Recall that for depth three MHS modulo $p$ can be determined
\cite[Theorem~3.5]{1stpart}, \cite[Theorem~6.2]{Hmodp}
and \cite[(3.13)]{1stpart}.
For AMHS, we first observe that for any $\ga,\gb,\gam\in \ZZ$
\begin{align*}
 H(\ga,\gb,\gam)=&H(\ga)H(\gb)H(\gam)-H(\gam)H(\gb,\ga)-H(\gam)H(\gb\oplus\ga)\\
 -&H(\gam,\gb)H(\ga)-H(\gam\oplus\gb)H(\ga)+H(\gam,\gb,\ga)\\
 +&H(\gam\oplus\gb,\ga)+H(\gam,\gb\oplus \ga)+H(\gam\oplus\gb\oplus\ga).
 \end{align*}
This can be easily checked by stuffle relations but the idea is hidden in the
general framework set up by Hoffman \cite{H}.
Combining with the reversal relations we can obtained the following results
without much difficulty. We leave its proof to the interested reader.
\begin{thm}
Let $p$ be a prime and $a,b,c\in \NN$ such that
$p>w$ where $w:=a+b+c$. Write $H(-)=H(-;p-1)$. Then we have

(1). If $w$ is even then
\begin{alignat}{2}
2 H(a,-b,c)\equiv & H(-c-b,a)+H(c,-b-a)  &\pmod{p}, \label{equ:a-bcEven}\\
2 H(a,b,-c)\equiv & -H(-c)H(b,a)+H(-c-b,a)+H(-c,b+a) &\pmod{p}, \label{equ:ab-cEven}\\
2 H(-a,-b,-c)\equiv &-H(-c)H(-b,-a)-H(-c,-b)H(-a)\\
 \ & \hskip2cm +H(c+b,-a)+H(-c,a+b) &\pmod{p}. \label{equ:-a-b-cEven}
\end{alignat}

(2). If $w$ is odd then
\begin{alignat}{2}
2 H(a,-b,-c)\equiv &H(c+b,a)+H(-c,-b-a)-H(-c)H(-b,a) &\pmod{p},\label{equ:a-b-cOdd} \\
2 H(-a,b,-c)\equiv &-H(-c)H(b,-a)-H(-c,b)H(-a)\\
 \ & \hskip2cm +H(-c-b,-a)+H(-c,-b-a) &\pmod{p}. \label{equ:-ab-cOdd}
\end{alignat}
\end{thm}

Because of the reversal relations when the weight is even there remains
essentially only one more case to consider in
depth three. This is given by the next result which will be used
in \S\ref{sec:wt4}.
\begin{thm}\label{thm:wtevendepth3}
Let $a,b,c$ be positive integers such  that $w:=a+b+c$ is even.
Then for any prime $p\ge w+3$  we have
\begin{multline*}
H(a,-b,-c)\equiv -\sum_{k=2}^{p-w+1}  {p-a\choose p-w-k+1}
     {k+c-1\choose c} \frac{(1-2^k)B_k B_{p-w-k+1}}{ak} \\
  - \sum_{k=p+1-b-c}^{p-c}  {p-a\choose 2p-w-k}
     {k+c-1\choose c} \frac{(1-2^{k})B_k B_{2p-w-k}}{ak} \\
  - \frac{(1-2^{p-c})(1-2^{p-a-b})B_{p-a-b}B_{p-c}}{(a+b)c} .
\end{multline*}
\end{thm}

\begin{proof} The proof is essentially a repeated application
of Theorem~\ref{thm:reduction}. But we spell out all the details below
because there are some subtle details that we need to attend to.

By \eqref{equ:sump} and Fermat's Little Theorem
we have modulo $p$
\begin{align*}
H(a,-b,-c)\equiv&\sum_{i=1}^{p-1} \frac{(-1)^i}{ i^{c}}\sum_{j=1}^{i-1}
    (-1)^j j^{p-1-b} \sum_{k=1}^{j-1} k^{p-1-a}\\
    \equiv &\sum_{i=1}^{p-1}  \frac{(-1)^i}{ i^{c}}
     \sum_{k=0}^{p-1-a}  {p-a\choose k} \frac{B_k}{p-a}\sum_{j=1}^{i-1}  (-1)^j j^{\rho(k)+p-a-b-k},
\end{align*}
where $\rho(k)=0$ if $k< p-a-b$ and $\rho(k)=p-1$ if $k\ge p-a-b$
(to make sure all exponents are positive in the sum of the second line above).
By Lemma \ref{lem:altsum}
\begin{align}
H(a,-b,-c)\equiv & \sum_{k=0}^{p-1-a}  {p-a\choose k} \frac{B_k}{p-a}
    \sum_{r=0}^n {n \choose r}  \sum_{i=1}^{p-1}  (-1)^i i^{n-r-c} F_{n,i,r} \notag\\
 \equiv &  -\sum_{k=0}^{p-1-a}  {p-a\choose k} \sum_{r=1}^{n-1}
     {n\choose r} \frac{(1-2^{r+1})B_k B_{r+1}}{(p-a)(r+1)}
      \sum_{i=1}^{p-1} i^{n-c-r}  \notag\\
  & +\sum_{k=0}^{p-1-a}  {p-a\choose k}  \frac{(1-2^{n+1})B_k B_{n+1}}{(p-a)(n+1)}
 \sum_{i=1}^{p-1} ((-1)^i-1) i^{-c}, \notag
\end{align}
where $n=\rho(k)+p-a-b-k$. Here we have used the fact that when $r=0$
we have $F_{n,i,r}=(-1)^{i-1}$ and thus the inner sum is
$\sum_{i=1}^{p-1} i^{n-c}\equiv 0 \pmod{p}$ except when $p-1|n-c$, i.e.,
except when $n=c$ and $k=p-w$. But then $B_k=0$ since $w$ is
even by assumption. Thus
\begin{align}
H(a,-b,-c)  \equiv & \sum_{\substack{ 0\le k< p-a\\ c<n}}  {p-a\choose k}
     {n\choose n-c} \frac{(1-2^{n-c+1})B_k B_{n-c+1}}{(p-a)(n-c+1)} \notag\\
  & +H(-c)\sum_{k=0}^{p-1-a}  {p-a\choose k}
  \frac{(1-2^{n+1})B_k B_{n+1}}{(p-a)(n+1)}. \notag
\end{align}
Now if $c$ is even then $H(-c)\equiv 0 \pmod{p}$. So we may
assume $c$ is odd in the last line above. Then $k+n+1$ is always
odd so that $B_k B_{n+1}\ne 0$ if and only if $k=1$ and $n=p-a-b-1$.
Hence
\begin{multline*}
H(a,-b,-c)\equiv \sum_{k=0}^{p-w-1}  {p-a\choose k}
     {p-a-b-k\choose p-w-k} \frac{(1-2^{p-w-k+1})B_k B_{p-w-k+1}}{(p-a)(p-w-k+1)} \\
  +\sum_{k=p-a-b}^{p-1-a}  {p-a\choose k}
     {2p-1-a-b-k\choose 2p-1-w-k} \frac{(1-2^{2p-w-k})B_k B_{2p-w-k}}{(p-a)(2p-w-k)} \\
  - \frac{(1-2^{p-c})(1-2^{p-a-b})B_{p-a-b}B_{p-c}}{(a+b)c} .
\end{multline*}
After substitutions $k\to p-w+1-k$ in the first sum and
$k\to 2p-w-k$ in the second sum the theorem follows
immediately from Cororllary~\ref{cor:Hdepth1},
\end{proof}

\begin{rem} The condition $p\ge w+3$ in Theorem \ref{thm:wtevendepth3}
can not be weakened since
\begin{equation*}
H(1,-2,-3)\equiv  \text{RHS}+5\not\equiv \text{RHS} \pmod{7}.
\end{equation*}
\end{rem}

\section{AMHS of arbitrary depth}\label{sec:depthAll}
In this section we provide some general results on AMHS without
restrictions on the depth. We first consider the homogeneous
case for which the key idea comes from
\cite[Lemma 2.12]{1stpart} and \cite[Theorem~2.3]{Hmodp}.

Let $p_i=\sum_{j\ge 1} x_j^i$ be the power-sum symmetric functions
and $e_i=\sum_{j_1<\cdots <j_i}x_{j_i} \cdots x_{j_i}$ be  the
elementary symmetric functions of degree $i$. Let $P(\ell)$ be the
set of unordered partitions of $\ell$. For
$\gl=(\gl_1,\dots,\gl_r)\in P(\ell)$ we set $p_\gl=\prod_{i=1}^r
p_{\gl_i}$. Recall that the expression of $e_l$ in terms of $p_i$
is given by the following formula (see \cite[p.28]{Mac}):
\begin{equation}\label{equ:pe}
\ell!e_\ell=\left|
\begin{matrix}
p_1 &1 &0  &\cdots & 0 \\
p_2 &p_1 &2  &\cdots & 0 \\
\ &\ & \ &\ddots & 0 \\
p_{\ell-1} &p_{\ell-2} &p_{\ell-3}  &\cdots & \ell-1 \\
p_\ell &p_{\ell-1} &p_{\ell-2}  &\cdots & p_1
\end{matrix}
\right|=\sum_{\gl\in P(\ell)} c_\gl p_\gl.
\end{equation}
Denote by $O(\ell)\subset P(\ell)$ the subset of odd partitions
$\gl=(\gl_1,\dots,\gl_r)$ (i.e., $\gl_i$ is odd for every part).

\begin{lem}\label{lem:cgl}
Let $a,\ell\in \NN$ and $p$ a prime such that $a\ell<p-1$.
Set $H(-)=H(-;p-1)$. For an odd partition
$\gl=(\gl_1,\dots,\gl_r)\in O(\ell)$
we put $H_\gl(-a)=\prod_{i=1}^r H(-\gl_i a)$. Then
\begin{equation}\label{cgl}
\ell! H\big(\{-a\}^\ell \big)\equiv
\sum_{\gl\in O(\ell)} c_\gl H_\gl(-a ) \pmod{p},
\end{equation}
where $c_\gl$ are given by \eqref{equ:pe}. In particular, if $a$ is even
then $\eqref{cgl}\equiv 0 \pmod{p}$. If $a>1$ is odd then
$\eqref{cgl}$ is congruent to a $\Q$-linear combination of
$B_{p-\gl_1 a}\cdots B_{p-\gl_r a}$ for odd partitions
$\gl=(\gl_1,\dots,\gl_r)$.
If $a=1$ then
$\eqref{cgl}$ is congruent to a $\Q$-linear combination of
$q_p^sB_{p-\gl_{s+1}}\cdots B_{p-\gl_r}$ for odd partitions
$\gl=(\{1\}^s,\gl_{s+1},\dots,\gl_r)$ ($\gl_i>1$ for all $i>s$),
where $q_p=(2^{p-1}-1)/p$ is the Fermat quotient.
\end{lem}

\begin{proof} Congruence \eqref{cgl} follows from
\cite[Lemma 2.12]{1stpart} and \cite[Theorem~2.3]{Hmodp}.
The last part follows from \cite[Theorem~2.1]{Tau1} and \eqref{equ:SunThm5.2}.
\end{proof}

For example, by  \cite[Theorem~2.1]{Tau1} and \eqref{equ:SunThm5.2}
we have
\begin{equation}\label{equ:H-1}
H(-1)=-H(1)+H(1;\frac{p-1}{2})\equiv
-2q_p+ pq_p^2 - \frac 23 p^2q_p^3 - \frac{1}{4}p^2 B_{p-3}\pmod{p^3}.
\end{equation}
Observe that $O(2)=\{(1,1)\}, O(3)=\{(1,1,1),(3)\}, O(4)=\{(1,1,1,1),(1,3)\}.$
It is obvious that $c_{(1,\cdots,1)}=1$, $c_{(3)}=2$,
$c_{(1,3)}=8,$ $c_{(1,1,3)}=20,$  and $c_{(1,1,1,3)}=c_{(3,3)}=40,$
which implies that
\begin{align} \label{equ:depth2FermatQ}
2 H\big(\{-1\}^2\big)\equiv & 4q_p^2,
\ \ \   \qquad \qquad \text{so }
H\big(\{-1\}^2\big)\equiv 2q_p^2 \\
6 H\big(\{-1\}^3\big)\equiv & -8q_p^3+2H(-3), \ \ \   \quad \quad \text{so }
H\big(\{-1\}^3\big)\equiv -\frac43q_p^3-\frac16B_{p-3} \label{equ:depth3FermatQ}\\
24H\big(\{-1\}^4\big) \equiv &16 q_p^4+8H(-1)H(-3),
\quad \text{so }
H\big(\{-1\}^4\big)\equiv \frac23q_p^4+\frac13 q_p B_{p-3},\\
5! H\big(\{-1\}^5\big) \equiv & -32 q_p^5+20H(-1)^2 H(-3),
\quad \text{so }
H\big(\{-1\}^5\big)\equiv -\frac4{15}q_p^5-\frac13 q_p^2 B_{p-3}
\end{align}
and
\begin{align}
6! H\big(\{-1\}^6\big) \equiv & 64 q_p^6+40H(-1)^3 H(-3)+40H(-3)^2, \notag \\
H\big(\{-1\}^6\big) \equiv & \frac4{45} q_p^6+\frac29 q_p^3 B_{p-3}+\frac1{72}B_{p-3}^2.
\end{align}

\begin{rem} Congruences \eqref{equ:depth2FermatQ} and
\eqref{equ:depth3FermatQ} are not new. See the Remarks
on \cite[p.\ 365]{CD}.
\end{rem}

For non-homogeneous $\ors$ we don't know too much except for those of very special
forms. For example we have the follow easy statement.
\begin{prop} Suppose $\ors=\ola{\ors}$ is palindromic. If the number of negative
components in $\ors$ and the weight $|\ors|$ have different parity then
\begin{equation*}
     S(\ors;p-1)\equiv H(\ors;p-1)\equiv 0 \pmod{p}.
\end{equation*}
\end{prop}
\begin{proof} This is obvious by the reversal relations \eqref{equ:reversal}
\end{proof}
In order to state and prove Proposition \ref{prop:-11r}, we need to investigate the
following two types of sums. Define
\begin{align}\label{equ:Udefn}
  U(\ors;n):=&\sum_{1\leq k_1<k_2<\dots<k_\ell\leq n}\;
\prod_{i=1}^\ell \frac{(-\sign(s_i)/2+3/2)^{k_i}}{k_i^{|s_i|}},\\
  V(\ors;n):=&\sum_{1\leq k_1<k_2<\dots<k_\ell\leq n}\;
\prod_{i=1}^\ell \frac{(\sign(s_i)/4+3/4)^{k_i}}{k_i^{|s_i|}}.\label{equ:Vdefn}
\end{align}
For example, $U(-m;n)=\sum_{k=1}^n 2^k/k^m$,
$V(-m;n)=\sum_{k=1}^n 1/(2^kk^m),$ and for $\ors$ with only
positive components $U(\ors;n)=V(\ors;n)=H(\ors;n)$. The sums $U(-m;p-1)$ appeared
previously in congruences involving powers of Fermat quotient
(see \cite{Agoh,Bach,DS,Gran}).
To compute the congruence involving these sums we need the following
preliminary results which will also be needed in \S\ref{sec:highPower}.
\begin{lem} \label{lem:Ureverse}
Assume $\ors=(s_1,\dots,s_\ell)\in (\ZZ^*)^{\ell}$. Let
$\bfe_j$ be the standard $j$-th unit vector. Then we have
\begin{align} \label{equ:Urev}
  U(\ors;n)\equiv& 2^{p\sharp\{j:s_j<0\}}(-1)^{|\ors|} \Big(V(\ola{\ors};n)+
p\sum_{j=1}^\ell |s_j| V(\ola{\ors\oplus \bfe_j} ;n) \Big)\pmod{p^2}\\
  V(\ors;n)\equiv& 2^{-p\sharp\{j:s_j<0\}}(-1)^{|\ors|} \Big(U(\ola{\ors};n)+
p\sum_{j=1}^\ell |s_j| U(\ola{\ors\oplus \bfe_j} ;n) \Big)\pmod{p^2}.\label{equ:Vrev}
\end{align}
Further,
\begin{equation}\label{equ:Vreverse}
V(-1)\equiv -\frac1{2^p} \Big( U(-1)+pU(-2)+p^2 U(-3)\Big) \pmod{p^3}.
\end{equation}
\end{lem}

\begin{lem}  \label{lem:mygeneral}
Let $p$ be an odd prime. For all positive integers $d$ and $m$  we have
\begin{align}
\sum_{1\le n_1\le \cdots\le n_d\le m} \frac{(1-x)^{n_1}-1}{n_1\cdots n_d}
=& \sum_{j=1}^{m}  \frac{(-x)^j}{j^d} {m\choose j} \label{equ:mygeneral(1)}\\
\sum_{1\le n_1\le \cdots\le n_d\le p-1} \frac{(1-x)^{n_1}-1}{n_1\cdots n_d}
\equiv &\sum_{j=1}^{p-1}  \frac{x^j}{j^d} \Big(1-pH(1;j)+p^2 H(1,1;j)\Big) \pmod{p^3}. \label{equ:mygeneral(2)}
\end{align}
\end{lem}
\begin{proof} Since for all positive integer $j<p$
\begin{equation}\label{equ:p-1choosej}
 (-1)^j {p-1\choose j}\equiv 1-pH(1;j)+p^2 H(1,1;j)\pmod{p^3},
\end{equation}
congruence \eqref{equ:mygeneral(2)} follows easily from \eqref{equ:mygeneral(1)}
which we now prove by mimicking the argument in \cite{monthly}.

Let $\QQ^*=\QQ\setminus\{0\}$.  Define the injective operator
\begin{align*}
\gD:\QQ[x]\setminus \QQ^*  & \lra \QQ[x]\setminus \QQ^* \\
 p(x)\ & \lmaps \ xp'(x).
\end{align*}
It suffices to show that
\begin{equation*}
\gD^d\left(\sum_{1\le n_1\le \cdots\le n_d\le m}
    \frac{(1-x)^{n_1}-1}{n_1\cdots n_d}\right)
=\gD^d\left(\sum_{j=1}^{m}  \frac{(-x)^j}{j^d} {m\choose j} \right).
\end{equation*}
Clearly
\begin{equation*}
\gD^d\left(\sum_{j=1}^{m}  \frac{(-x)^j}{j^d} {m\choose j} \right)
= \sum_{j=1}^{m}  (-x)^j  {m\choose j}=(1-x)^m-1.
\end{equation*}
On the other hand,
\begin{align*}
\ &\gD^d\left(\sum_{1\le n_1\le n_2\le \cdots\le n_d\le m}
    \frac{(1-x)^{n_1}-1}{n_1n_2\cdots n_d}\right)\\
\ =&\gD^{d-1}\left(\sum_{1\le n_2\le \cdots\le n_d\le m}
    \frac{(1-x)^{n_2}-1}{n_2\cdots n_d}\right)\\
\  & \vdots \\
\ =&\gD\left(\sum_{1\le n_d\le m} \frac{(1-x)^{n_d}-1}{ n_d}\right)\\
\ =&-x \sum_{1\le n_d\le m} (1-x)^{n_d-1} \\
\ =& (1-x)^m-1.
\end{align*}
By injectivity of $\gD$ this completes the proof of the lemma.
\end{proof}

\begin{lem}\label{lem:binomialU}
Let $d$ and $m$ be two positive integers. Then
 \begin{equation}\label{equ:binomialU}
\sum_{1\le n_1\le \cdots\le n_d\le m} \frac{(-1)^{n_d} (1-x)^{n_1}}{n_1\cdots n_d} {m\choose n_d}
=\sum_{k=1}^m \frac{x^k}{k^d}-\sum_{k=1}^m \frac{1}{k^d}.
\end{equation}
\end{lem}
\begin{proof} The proof, which is completely similar to \cite{monthly}, is left to the
interested reader. Note that there is a misprint in \cite{monthly} where in the definition
of $f(x)$ the range of $i$ should be from $1$ to $j$. Namely
\begin{equation*}
 f(x)=\sum_{k=1}^n \frac{(-1)^{k-1}}{k}{n\choose k} \sum_{j=1}^{k} \frac1j \sum_{i=1}^j \frac{(1-x)^i}{i}.
\end{equation*}
\end{proof}

\begin{prop} \label{prop:-11r}
Let $n\in\NN$. For any prime $p>n+2$ write $H(-)=H(-;p-1)$ and similarly
for $S,U$ and $V$.  Then
\begin{multline} \label{equ:-11n}
H(-1,\{1\}^n)\equiv S(-1,\{1\}^n)\equiv (-1)^n H(\{1\}^n,-1)\equiv (-1)^n S(\{1\}^n,-1)\\
\equiv U(-n-1)\equiv   (-1)^{n+1} 2  V(-n-1) \pmod{p}.
\end{multline}
\end{prop}
\begin{proof}
The congruence $(-1)^n H(\{1\}^n,-1)\equiv U(-n-1)\pmod{p}$ is the content of
\cite[Theorem~2.3]{Tau1}. By taking $m=p-1$, $d=n+1$ and $x=2$
in Lemma~\ref{lem:mygeneral} or Lemma~\ref{lem:binomialU} we get
$S(-1,\{1\}^n)\equiv U(-n-1)\pmod{p}$. The congruence for $V$ follows from
\eqref{equ:Urev}. The other two congruences follows from the
reversal relations~\eqref{equ:reversal}.
\end{proof}

We now generalize this to the following
\begin{prop} \label{prop:am-aan}
Let $m,n$ be two nonnegative integers and $a$ a positive integer.
For any prime $p>a(m+n)+2$ write $H(-)=H(-;p-1)$.  Then
\begin{alignat}{2}
 \label{equ:am-aan1}
H(\{a\}^m,-a,\{a\}^n)\equiv& (-1)^{m+n} S(\{a\}^n,-a,\{a\}^m) &\pmod{p},\\
 \label{equ:am-aan2}
H(\{a\}^m,-a,\{a\}^n)\equiv& (-1)^{(m+n+1)(a+1)} S(\{a\}^m,-a,\{a\}^n) &\pmod{p}.
\end{alignat}
\end{prop}
\begin{proof} The first congruence \eqref{equ:am-aan1} follows from \eqref{equ:HS1}
in the next Lemma by taking $x=-1$, $k=m+1$ and $d=m+n+1$. The second congruence
\eqref{equ:am-aan2} follows \eqref{equ:am-aan1} by the reversal relation.
\end{proof}

Let $a,d\in \NN$ and a prime $p\ge da+3$. We identify
the finite field $\bbF_p$ of $p$ elements with $\bbZ/p\bbZ$.
For $1\leq k \leq d$ define
\begin{align*}
H^{(a)}_{d,k}(x)=&\sum_{0<i_1<\cdots<i_d<p} \frac{x^{i_k}}{(i_1\cdots i_d)^a}\in \bbF_p[x],\\
S^{(a)}_{d,k}(x)=&\sum_{1\leq i_1\leq\cdots\leq i_d<p} \frac{x^{i_k}}{(i_1\cdots i_d)^a}\in \bbF_p[x]\\
h^{(a)}_{d,k}(x)=&\sum_{0<i_1<\cdots<i_d<p} \frac{(-1)^{\sum i_j} x^{i_k}}{(i_1\cdots i_d)^a}\in \bbF_p[x],\\
s^{(a)}_{d,k}(x)=&\sum_{1\leq i_1\leq\cdots\leq i_d<p} \frac{(-1)^{\sum i_j} x^{i_k}}{(i_1\cdots i_d)^a}\in \bbF_p[x],
\end{align*}
where $\sum i_j=i_1+\cdots+ i_d$.
For convenience we set $H^{(a)}_{d,0}(x)=H(\{a\}^d)=0$, $S^{(a)}_{d,d+1}(x)=x^{p-1} S(\{a\}^d)=0$,
$S^{(a)}_{d,0}(x)=xS(\{a\}^d)=0$ and $H^{(a)}_{d,d+1}(x)=x^p H(\{a\}^d)=0$ by \cite[Theorem~2.13]{1stpart}.
Moreover for even $a$ we set $h^{(a)}_{d,0}(x)=H(\{-a\}^d)=0$, $s^{(a)}_{d,d+1}(x)=(-x)^{p-1} S(\{-a\}^d)=0$,
$s^{(a)}_{d,0}(x)=-x S(\{-a\}^d)=0$ and $h^{(a)}_{d,d+1}(x)=(-x)^p H(\{-a\}^d)=0$ by Lemma~\ref{lem:cgl}
and \eqref{equ:Hdepth1}.

\begin{lem}
For $1\leq k \leq d$ we have the identity in $\bbF_p[x]$
\begin{align}\label{equ:HS1}
 H^{(a)}_{d,k}(x)+(-1)^{d} S^{(a)}_{d,d+1-k}(x)= 0,\\
 h^{(a)}_{d,k}(x)+(-1)^{d} s^{(a)}_{d,d+1-k}(x)= 0.\label{equ:HS2}
\end{align}
\end{lem}

\begin{proof}
To prove \eqref{equ:HS1} we proceed by induction on $d$. For $d=1$ it holds since
$$H^{(a)}_{1,1}(x)=\sum_{k=1}^{p-1} \frac{x^k}{k^a}=S^{(a)}_{1,1}(x),\quad
H^{(a)}_{1,0}(x)=S^{(a)}_{1,0}(x)=H^{(a)}_{1,2}(x)=S^{(a)}_{1,2}(x)=0.$$
Now we will follow the idea of the proof of  \cite[Theorem~2.3]{Tau1}.
For $1\leq k \leq d$ we have that
\begin{align*}
\left(x\frac{d}{dx}\right)^a H^{(a)}_{d,k}(x)
=&\sum_{0<i_1<\dots<i_{k-1}<i_{k+1}<\dots<i_{d}<p} \frac{1}{(i_1\cdots i_{k-1}i_{k+1}\cdots  i_{d})^a}
\sum_{i_k=i_{k-1}+1}^{i_{k+1}-1} x^{i_k}\\
=&\sum_{0<i_1<\dots<i_{k-1}<i_{k+1}<\dots<i_{d}<p}
\frac{1}{(i_1\cdots i_{k-1}i_{k+1}\cdots  i_{d})^a}\cdot \frac{x^{i_{k+1}}-x^{i_{k-1}+1}}{ x-1}\\
=&\frac{1}{x-1}\, H^{(a)}_{d-1,k}(x)-\frac{x}{x-1}\, H^{(a)}_{d-1,k-1}(x),
\end{align*}
that is
$$(x-1) \left(x\frac{d}{dx}\right)^a
 H^{(a)}_{d,k}(x) =H^{(a)}_{d-1,k}(x)-xH^{(a)}_{d-1,k-1}(x).$$
In a similar way we have that
$$(x-1) \left(x\frac{d}{dx}\right)^a S^{(a)}_{d,k}(x)=xS^{(a)}_{d-1,k}(x)-S^{(a)}_{d-1,k-1}(x).$$
Hence
\begin{align*}
& (x-1) \left(x\frac{d}{dx}\right)^a \left(H^{(a)}_{d,k}(x)+(-1)^{d} S^{(a)}_{d,d+1-k}(x)\right)\\
\ =&
H^{(a)}_{d-1,k}(x)-xH^{(a)}_{d-1,k-1}(x)
+(-1)^{d} xS^{(a)}_{d-1,d+1-k}(x)-(-1)^{d}S^{(a)}_{d-1,d-k}(x)\\
=&H^{(a)}_{d-1,k}(x)+(-1)^{d-1}S^{(a)}_{d-1,d-k}(x)
-x\left(H^{(a)}_{d-1,k-1}(x)+(-1)^{d-1}S^{(a)}_{d-1,d-(k-1)}(x)\right)= 0.
\end{align*}
Thus $\left(x\frac{d}{dx}\right)^{a-1}\Big( H^{(a)}_{d,k}(x)+(-1)^{d} S^{(a)}_{d,d+1-k}(x)\Big)= c$
for some constant $c\in \bbF_p$ since this polynomial has degree less than $p$.
By letting $x=0$ we see that $c=0$. Repeating this process $a$ times yields \eqref{equ:HS1}.

The congruence \eqref{equ:HS2} can be proved in a similar way. In particular, we can
show easily that
\begin{align*}
 -(x+1) \left(x\frac{d}{dx}\right)^a
 h^{(a)}_{d,k}(x) = & h^{(a)}_{d-1,k}(-x)+xh^{(a)}_{d-1,k-1}(-x),\\
 (x+1)\left(x\frac{d}{dx}\right)^a s^{(a)}_{d,k}(x)= & xs^{(a)}_{d-1,k}(-x)+s^{(a)}_{d-1,k-1}(-x).
\end{align*}
So by induction we get
$$(x+1) \left(x\frac{d}{dx}\right)^a \left(H^{(a)}_{d,k}(x)+(-1)^{d} S^{(a)}_{d,d+1-k}(x)\right)=0,$$
which quickly leads to \eqref{equ:HS2}. This concludes the proof of the lemma.
\end{proof}

Similar to Proposition \ref{prop:am-aan} we also have the following
\begin{prop} \label{prop:-ama-an}
Let $m,n$ be two nonnegative integers. Let $a$ be a positive even integer.
For any prime $p>a(m+n)+2$ write $H(-)=H(-;p-1)$.  Then
\begin{align}
\label{equ:-ama-an1}
H(\{-a\}^m,a,\{-a\}^n)\equiv& S(\{-a\}^m,a,\{-a\}^n)& \pmod{p},\\
\label{equ:-ama-an2}
H(\{-a\}^m,a,\{-a\}^n)\equiv& (-1)^{(m+n)} S(\{-a\}^n,a,\{-a\}^m)& \pmod{p}.
\end{align}
\end{prop}
\begin{proof} The second congruence \eqref{equ:-ama-an1} follows from \eqref{equ:HS2}
by taking $x=-1$, $k=m+1$ and $d=m+n+1$. The first congruence
\eqref{equ:-ama-an2} follows \eqref{equ:-ama-an1} by the reversal relation
since $a$ is even.
\end{proof}

\section{AMHS of weight four}\label{sec:wt4}
In \cite{Tau1} the first author studied the congruence properties
of AMHS of weight less than four. In this section, applying the
results obtained in the previous sections we can analyze the weight
four AMHS in some detail. First we treat some special congruences
which can not be obtained by just using the stuffle relations and
the reversal relations. Let $p\ge 7$ be a prime and
set $A=A_p,\cdots, K=K_p$ as follows:
\begin{alignat*}{3}
&A:=\sum_{k=2}^{p-3} B_kB_{p-3-k},\quad
&B:=\sum_{k=2}^{p-3} 2^k B_kB_{p-3-k} ,\quad\,
&C:=\sum_{k=2}^{p-3} 2^{p-3-k} B_kB_{p-3-k} ,\\
&D:=\sum_{k=2}^{p-3} \frac{B_kB_{p-3-k}}{k},\quad
&E:=\sum_{k=2}^{p-3} \frac{2^k B_kB_{p-3-k}}{k},\quad\,
&F:=\sum_{k=2}^{p-3} \frac{2^{p-3-k} B_kB_{p-3-k}}{k},\\
&G:=\sum_{k=2}^{p-3} k B_kB_{p-3-k},\quad
&J:=\sum_{k=2}^{p-3} 2^k k B_kB_{p-3-k},\quad
&K:=\sum_{k=2}^{p-3} 2^{p-3-k} k B_kB_{p-3-k}.
\end{alignat*}
Then by \cite[Cororllary~3.6]{1stpart} and simple computation
\begin{equation}\label{equ:relationsA-J}
A\equiv -B_{p-3}, \quad G\equiv 0, \quad
C\equiv B-\frac34 A, \quad  K\equiv -3B-J+3A \pmod{p}.
\end{equation}

\begin{prop} \label{prop:wt4depth2}
For all prime $p\ge 7$ write $H(-)=H(-;p-1)$. Then  we have
\begin{alignat}{2}
 H(1,-3)\equiv&  \frac12 H(-2,2) \equiv B-A\equiv
    \sum_{k=0}^{p-3} 2^k B_kB_{p-3-k} &\pmod{p}, \label{equ:2H1-3=H-22}\\
     H(-1,3)\equiv&  -\frac12 q_pB_{p-3} &\pmod{p}.\label{equ:H-13}
\end{alignat}
\end{prop}
\begin{proof}
We take congruence modulo $p$ throughout this proof.
By Proposition~\ref{prop:wteven2} we have
\begin{align*}
H(-3,1) \equiv
 &-\frac13\sum_{k=1}^{p-6}(k+2)(k+3)(1-2^{k+1})(1-2^{p-4-k})
    \frac{B_{k+1} B_{p-4-k}}{k+4}-\frac12 q_pB_{p-3}   \\
 \equiv& \frac13\sum_{k=2}^{p-3}(k+1)(k+2)(1-2^k)(1-2^{p-3-k})
    \frac{B_k B_{p-3-k}}{k}-\frac12 q_pB_{p-3}.
\end{align*}
by the substitution $k\to p-4-k$. Similarly we can get
\begin{align*}
H(-2,2)  \equiv&-\sum_{k=2}^{p-3}(k+2)
     (1-2^k)(1-2^{p-3-k})\frac{{B_k} B_{p-3-k}}{k}
 +\frac32 q_pB_{p-3}  \\
H(-1,3)  \equiv&2\sum_{k=2}^{p-3}
     (1-2^k)(1-2^{p-3-k})\frac{{B_k} B_{p-3-k}}{k}-2q_pB_{p-3}
\end{align*}
Using \eqref{equ:relationsA-J} we reduce the above to
\begin{align}
3 H(-3,1) \equiv & 3A-3B+ \frac52D-2E-2F-\frac32 q_pB_{p-3} \label{equ:-31I}  \\
H(-2,2) \equiv&   2B-2A-\frac52D+2E+2F+\frac32 q_pB_{p-3}  \label{equ:-22I}  \\
H(-1,3) \equiv& \phantom{2B-2A}+\frac52D-2E-2F-2q_pB_{p-3} \label{equ:-13I}
\end{align}
On the other hand, by Proposition~\ref{prop:wteven} we get
\begin{align*}
H(1,-3)\equiv& \frac{2}{p-1}\sum_{k=0}^{p-2}{p-1\choose k}
\frac{1-2^{2p-4-k}}{2p-4-k}B_kB_{2p-4-k} \\
\equiv&  -2\sum_{k=0}^{p-5}
\frac{1-2^{p-3-k}}{p-3-k}B_kB_{p-3-k}+2q_pB_{p-3} \\
\equiv&  -2\sum_{k=2}^{p-3}
\frac{1-2^k}{k}B_kB_{p-3-k}+2q_pB_{p-3} \\
\equiv&   2E-2D+2q_pB_{p-3}.
\end{align*}
by the substitution $k\to p-3-k$. Thus by the reversal relation
\begin{equation} \label{equ:-31II}
 H(-3,1) \equiv -H(1,-3)\equiv  2D-2E-2q_pB_{p-3}.
\end{equation}
Similarly we find
\begin{align}
H(-2,2) \equiv& -H(2,-2)\equiv  B-A+2E-2D+2q_pB_{p-3}, \label{equ:-22II}  \\
H(-1,3) \equiv& -H(3,-1)\equiv \frac{1}{3}\Big(-J+3A-3B+2D-2E\Big).\label{equ:-13II}
\end{align}
Then by adding \eqref{equ:-31I}, \eqref{equ:-22I},
\eqref{equ:-31II} and \eqref{equ:-22II} altogether we have
\begin{equation*}
4H(-3,1)+2H(-2,2)\equiv  0
\end{equation*}
which implies the first congruence in \eqref{equ:2H1-3=H-22}. Now adding
\eqref{equ:-31II} and \eqref{equ:-22II} yields
\begin{equation}\label{equ:addagainH(-3,1)}
 -H(-3,1)\equiv H(-3,1)+H(-2,2)\equiv B-A
\end{equation}
which is the second congruence in \eqref{equ:2H1-3=H-22}.
Plugging this into \eqref{equ:-31I} we see that
\begin{equation*}
\frac52D-2E-2F-\frac32 q_pB_{p-3}\equiv 0
\end{equation*}
which combined with  \eqref{equ:-13I} produces \eqref{equ:H-13}.
This finishes the proof of the proposition.
\end{proof}

\begin{prop} \label{prop:wt4depth3}
For all prime $p\ge 7$ write $H(-)=H(-;p-1)$. Then  we have
\begin{alignat}{3}
H(1,-1,-2)\equiv& \frac12 H(1,-3)+\frac12  J& \ &\pmod{p}, \label{equ:2H1-1-2=H1-3}\\
H(1,-2,-1)\equiv& \phantom{-} H(1,-3) &-\frac54 q_pB_{p-3} &\pmod{p}, \label{equ:H1-2-1=H1-3}\\
H(2,-1,-1)\equiv& -H(1,-3)-\frac12 J&+\frac34 q_pB_{p-3} &\pmod{p}.\label{equ:H2-1-1=H1-3}
\end{alignat}
\end{prop}
\begin{proof} By Theorem~\ref{thm:wtevendepth3}, Proposition~\ref{prop:wt4depth2} and
\eqref{equ:relationsA-J} we get
\begin{align}
2H(1,-1,-2)\equiv& -\sum_{k=2}^{p-3}(k+1) (1-2^k)B_kB_{p-3-k}
\equiv B-A+J-G=H(1,-3)+J \label{equ:H1-1-2Use} \\
H(1,-2,-1)\equiv& -\sum_{k=2}^{p-3}(1-2^k)B_kB_{p-3-k}-\frac54 q_pB_{p-3}
\equiv H(1,-3)-\frac54 q_pB_{p-3} \notag \\
H(2,-1,-1)\equiv& \sum_{k=2}^{p-3} \frac{(k+2)(1-2^k)B_kB_{p-3-k}}{-2}+\frac34 q_pB_{p-3}
\equiv -H(1,-3)-\frac12 J+\frac34 q_pB_{p-3}, \notag
\end{align}
as claimed.
\end{proof}

By \eqref{equ:ab-cEven} and \eqref{equ:addagainH(-3,1)} it is readily seen that
\begin{equation*}
2H(1,1,-2)\equiv H(-3,1)+H(-2,2)\equiv -H(-3,1).
\end{equation*}
In fact, by the stuffle and the reversal relations
we can find congruences of all weight four AMHS of depth up to three.
By the reversal relations we only need to list about half of the values.
\begin{prop}\label{prop:allwt4}
For all prime $p\ge 7$ write $H(-)=H(-;p-1)$ and set
$H_{\bar{2}11}:=H(-2,1,1)$. Then $H_{\bar{2}11}=(A-B)/2$ and
{$$
\begin{array}{rll}
& H(4) \equiv H(-4) \equiv H(2,2)\equiv H(-2,-2) \equiv
    &\ \hskip-22pt H(1,3)\equiv H(1,-2,1) \equiv H(-1,-2,-1) \equiv 0, \\
& H(1,-3) \equiv -2 H_{\bar{2}11},\quad\
H(2,-2) \equiv 4 H_{\bar{2}11},\quad &
H(1,-1,2) \equiv 3 H_{\bar{2}11},\\
& H(-1,-3) \equiv\frac12 q_p B_{p-3},\quad
& H(3,-1) \equiv\frac12 q_p B_{p-3},\\
&H(1,-1,-2) \equiv -H_{\bar{2}11}+\frac12 J,
& H(-2,-1,-1) \equiv 2 H_{\bar{2}11}-q_p B_{p-3},\\
&H(-1,2,1) \equiv -H_{\bar{2}11}+\frac54q_p B_{p-3},\quad
& H(-1,1,2) \equiv 2 H_{\bar{2}11}-\frac34 q_p B_{p-3},\\
&H(-2,1,-1) \equiv 3 H_{\bar{2}11}-\frac12 J+\frac34 q_p B_{p-3},\quad
& H(1,-2,-1) \equiv -2 H_{\bar{2}11}-\frac54q_p B_{p-3},\\
&H(-1,2,-1) \equiv -4 H_{\bar{2}11}+J-\frac52q_p B_{p-3},\quad
& H(2,-1,-1) \equiv 2 H_{\bar{2}11}-\frac12 J+\frac34 q_p B_{p-3}.
\end{array}$$}
\end{prop}

We now turn to the depth four cases.

\begin{prop} \label{prop:wt4depth3Eg}
Let $p\ge 7$ be a prime and write $H(-)=H(-;p-1)$. Then we have
\begin{alignat}{5}
H(1,-1,-1,1)\equiv  &-&\frac12
    \Big(&H(1,-3)&+J & & &+q_p^4\Big) & \pmod{p}, \label{equ:H1-1-11}\\
H(-1,-1,1,1)\equiv H(1,1,-1,-1)\equiv& &\frac1{24}
    \Big(& &6J&+ 7&q_pB_{p-3}&+8q_p^4\Big) & \pmod{p}, \label{equ:H11-1-1} \\
H(-1,1,-1,1)\equiv H(1,-1,1,-1)\equiv &-&\frac1{12}
    \Big(& & & &q_pB_{p-3}&+2q_p^4\Big) & \pmod{p},  \label{equ:H1-11-1}\\
H(-1,1,1,-1)\equiv&  &\frac1{12}\Big(& 6H(1,-3)& &+7&q_pB_{p-3}&+2q_p^4\Big)
   &  \pmod{p}. \label{equ:H-111-1}
\end{alignat}
\end{prop}
\begin{proof} By Theorem~\ref{thm:reduction}
\begin{align*}
H(1,-1,-1,1)\equiv&  -H(-1,-1,1)-\frac12 H(-2,-1,1)-\sum_{k=2}^{p-3}  B_k H(-(k+1),-1,1),\\
H(1,1,-1,-1)\equiv& -H(1,-1,-1)-\frac12 H(2,-1,-1)-\sum_{k=2}^{p-3}  B_k H(k+1,-1,-1).
\end{align*}
Using reversal relations and \eqref{equ:a-b-cOdd} we see that
\begin{align}
H(1,-1,-1,1)\equiv& -H(-1,-1,1)-\frac12 H(-2,-1,1) \notag\\
    &+\frac12\sum_{k=2}^{p-3}  B_k  \Big(H(k+2,1)+H(-(k+1),-2)-H(-(k+1))H(-1,1)\Big),
    \label{equ:H1-1-11midstep1}\\
H(1,1,-1,-1)\equiv& -H(1,-1,-1)-\frac12 H(2,-1,-1)  \notag \\
    & -\frac12\sum_{k=2}^{p-3}  B_k  \Big(H(2,k+1)+H(-1,-(k+2))-H(-1)H(-1,k+1)\Big).
    \label{equ:H11-1-1midstep1}
\end{align}
Note that by Theorem~\ref{thm:reduction}
\begin{align*}
-H(-1,1) \equiv&  H(1,-1)  \equiv -\sum_{k=0}^{p-3}(-1)^k  B_k H(-(k+1)), \\
H(1,1,-1) \equiv& -H(1,-1) -\frac12 H(2,-1)+\sum_{k=2}^{p-3} B_k H (-1,k+1) .
\end{align*}
We may use \eqref{equ:l2wtodd1}, and \eqref{equ:l2wtodd2}
to simplify \eqref{equ:H1-1-11midstep1} and  \eqref{equ:H11-1-1midstep1}
further. For all $k=2,\dots,p-5$ we have
\begin{alignat}{2}
    H(k+2,1)\equiv& -B_{p-3-k},\quad
  & H(-(k+1),-2)\equiv\frac12(2^{p-3-k}-1)(k+2)B_{p-3-k}, \label{equ:stilltrue1}\\
H(2,k+1)\equiv& -\frac{(k+2)B_{p-3-k}}2,  \quad
& H(-1,-(k+2))\equiv (2^{p-3-k}-1) B_{p-3-k}. \label{equ:stilltrue2}
\end{alignat}
However, one has to be very careful in applying
these formulae because the formulae might fail when $k=p-3$.
We need to compute these separately as follows:
\begin{align*}
H(p-1,1)=&\sum_{i=1}^{p-1} \frac1i\sum_{j=1}^{i-1} \frac1{j^{p-1}}
 \equiv \sum_{i=1}^{p-1} \frac{i-1}i\equiv -1\equiv-B_0,\\
H(-(p-2),-2)=&\sum_{i=1}^{p-1} \frac{(-1)^i}{i^2}\sum_{j=1}^{i-1} \frac{(-1)^j}{j^{p-2}}
 \equiv \sum_{i=1}^{p-1} \frac{(-1)^i}{i^2}\left((-1)^i\frac{1-2i}4-\frac14\right) \equiv 0,\\
H(2,p-2)\equiv & H(p-2,2)=-\sum_{i=1}^{p-1} \frac{1}{i^2}\sum_{j=1}^{i-1} \frac{1}{j^{p-2}}
 \equiv -\sum_{i=1}^{p-1} \frac{i-1}{2i}
 \equiv \frac12 \equiv 0,\\
H(-1,-(p-1))\equiv& -H(-(p-1),-1)=-\sum_{i=1}^{p-1} \frac{(-1)^i}i
    \sum_{j=1}^{i-1} \frac{(-1)^j}{j^{p-1}}
 \equiv \sum_{i=1}^{p-1} \frac{1+(-1)^i}{2i}\equiv \frac12H(-1) \equiv -q_p
\end{align*}
by \eqref{equ:H-1}. We see that only $H(-1,-(p-1))$ fails the formula in
\eqref{equ:stilltrue2} and therefore  we get
\begin{align*}
H(1,-1,-1,1)\equiv& -H(-1,-1,1)-\frac12 H(-2,-1,1)+\frac12 H(-1)H(-1,1)- \frac12 H(-1,1)^2 \\
&+\frac12\sum_{k=2}^{p-3} B_k  \left(
-B_{p-3-k}+\frac12(2^{p-3-k}-1)(k+2)B_{p-3-k}\right)
\end{align*}
and
\begin{align*}
H(1,1,-1,-1)\equiv&  H(-1,-1,1)-\frac12 H(2,-1,-1)
    +\frac12 H(-1)\Big(H(1,1,-1)+H(1,-1)+\frac12 H(2,-1)\Big) \\
    &-\frac12\sum_{k=2}^{p-3} B_k  \left(
    -\frac{(k+2)B_{p-3-k}}2+(2^{p-3-k}-1)B_{p-3-k}\right)+\frac12 q_pB_{p-3}.
\end{align*}
Now by \cite[Cororllary~2.4, Cororllary~2.5]{Tau1} we know $H(-1,1)\equiv -q_p^2$,
$H(-1,-1,1)\equiv q_p^3+\frac78 B_{p-3}$, and
$H(1,1,-1)\equiv -\frac13q_p^3-\frac7{24} B_{p-3}$.
Together with \eqref{equ:H-1} these yield
\begin{align*}
H(1,-1,-1,1)\equiv & -\frac12 B_{p-3}-\frac12 H(-2,-1,1)-\frac12 q_p^4 -\frac12A
 +\frac14\sum_{k=0}^{p-3}   (2^{p-3-k}-1)(k+2) B_k B_{p-3-k},   \\
H(1,1,-1,-1)\equiv& B_{p-3}-\frac12 H(2,-1,-1)
+\frac13q_p^4+\frac23 q_pB_{p-3}+A+\frac14G
    -\frac12\sum_{k=0}^{p-3} 2^{p-3-k}B_k B_{p-3-k}.
\end{align*}
It now follows from \eqref{equ:relationsA-J} and the substitution $k\to p-3-k$ that
\begin{align*}
H(1,-1,-1,1)\equiv & -\frac12 H(-2,-1,1)-\frac12 q_p^4
 +\frac14\sum_{k=0}^{p-3}(1-2^k)(k+1) B_k B_{p-3-k}, \\
H(1,1,-1,-1)\equiv& -\frac12 H(2,-1,-1)-\frac12 H(1,-3)+\frac23 q_pB_{p-3-k}+\frac13 q_p^4.
\end{align*}
Hence \eqref{equ:H1-1-11} and \eqref{equ:H11-1-1} quickly follow from
\eqref{equ:H1-1-2Use} and  \eqref{equ:H2-1-1=H1-3}.

Finally, \eqref{equ:H1-11-1} follows from stuffle relations applied to
$H(-1)H(1,-1,1)$ and then \eqref{equ:H-111-1} from stuffle relations applied to
$H(-1)H(1,1,-1)$. This finishes the proof of the proposition.
\end{proof}

For other depth four and weight four AMHS we have the following relations
derived from the stuffle relations and the congruences obtained above:
\begin{align*}
H(1,1,-1,1) \equiv& 2H_{\bar{2}11}+3H(-1,1,1,1)+\frac12q_pB_{p-3}, \\
H(-1,-1,1,-1) \equiv&6H_{\bar{2}11}+3H(1,-1,-1,-1)-4q_pB_{p-3}-2q_p^4, \\
H(-1,-1,-1,-1)\equiv&\phantom{ 6H_{\bar{2}11}-3H(1,-1,-1,-1)+}  \frac13q_pB_{p-3}+\frac23q_p^4.
\end{align*}
On the other hand, we can only deduce from Theorem~\ref{thm:reduction}
and Theorem~\ref{thm:reduction2} that
\begin{align*}
H(1,1,1,-1)\equiv& -H(1,1,-1)-\frac12H(2,1,-1)
   -\sum_{k=2}^{p-3} B_k H(k+1,1,-1),\\
H(-1,-1,-1,1)\equiv& -q\Big(H(-1,-1,1)-H(1,-1,1)\Big)-\frac12H(2,-1,1)  \\
   &+\sum_{k=2}^{p-3} (1-2^k)B_k H(k+1,-1,1),
\end{align*}
where, by the reduction theorems again,
\begin{align*}
H(k+1,1,-1)\equiv&  -\frac{1}{k+1} H(k+1,-1)-\frac12 H(k+2,-1)\\
 &+\sum_{j=2}^{p-k-2}  {p-k-1\choose j} \frac{B_j}{p-k-1} H(j+k+1,-1),\\
H(k+1,-1,1)\equiv&  -\frac{1}{k+1} H(-(k+1),1)-\frac12 H(-(k+2),1)\\
 &+\sum_{j=2}^{p-k-2}  {p-k-1\choose j} \frac{B_j}{p-k-1} H(-(j+k+1),1).
\end{align*}
Observe that the indices $k$ and $j$ in the above sums can be both
taken to be even numbers. Thus by Proposition~\ref{prop:wteven}
and Proposition~\ref{prop:wteven2}
\begin{align*}
H(j+k+1,-1) \equiv
    &\sum_{i=0}^{p-j-k-2}{p-j-k-1\choose i}
    \frac{2(1-2^{2p-i-j-k-2})B_i B_{2p-i-j-k-2}}{(p-j-k-1)(2p-i-j-k-2)}, \\
\equiv
    &\sum_{i=0}^{p-j-k-2}{p-j-k-1\choose i}
    \frac{2(1-2^{p-i-j-k-1})B_i B_{p-i-j-k-1}}{(j+k+1)(i+j+k+1)},
\end{align*}
\begin{align*}
H(-(j+k+1),1) \equiv
&\sum_{i=1}^{p-j-k-4} {p-2-j-k\choose i}
    \frac{ 2 (1-2^{i+1})(1-2^{p-i-j-k-2})B_{i+1} B_{p-i-j-k-2}}{(i+1)(i+j+k+2)} \\
&+ \frac{2 (1-2^{p-j-k-1})(1-2^{p-1})B_{p-j-k-1} B_{p-1}}{p-j-k-1}\\
\equiv
&\sum_{i=2}^{p-j-k-3} {p-2-j-k\choose i}
    \frac{ 2 (1-2^i)(1-2^{p-i-j-k-1})B_i B_{p-i-j-k-1}}{i(i+j+k+1)} \\
&-q_p\frac{2 (1-2^{p-j-k-1})B_{p-j-k-1}}{j+k+1} .
\end{align*}
Consequently, both $H(1,1,1,-1)$ and $H(-1,-1,-1,1)$ can be written as
a triple sum with most of the terms involving products $B_iB_jB_k B_{p-i-j-k-2}$.
It is very likely that modulo $p$ we cannot reduce
$H(1,1,1,-1)$ and $H(-1,-1,-1,1)$ to a linear combination of AMHS
of depths up to three. At least in theory one possible way to
check this hypothesis is to find six infinite sets of primes
$S_1=\{p_1^{(k)}: k\ge 1\},\dots,S_6=\{p_6^{(k)}: k\ge 1\}$
for each of the following six elements:
\begin{alignat*}{3}
 b_1(p)=J_p, &\quad b_2(p)=H(1,-3), &\quad b_3(p)=&H(1,-1,-1,-1),\\
\ b_4(p)=q_p^4, &\quad b_5(p)=q_pB_{p-3}, &\quad b_6(p)=&H(-1,1,1,1),
\end{alignat*}
such that for each choice $(p_1^{(k)},\dots,p_6^{(k)})$ we always
have $b_j(p_i^{(k)})\equiv 0\pmod{p_i^{(k)}}$ for all $i\ne j$ and
$b_j(p_j^{(k)})\not\equiv 0\pmod{p_j^{(k)}}$ for all $j=1,\dots,6$.
In practice this is extremely difficult to carry out. For example,
if $b_4(p)\equiv 0\pmod{p}$ then the prime $p$ is called a Wieferich prime.
The only known Wieferich primes are 1093 and 3511
and if any other Wieferich primes exist, they must be greater than
$6.7 \times 10^{15}$ according to \cite{DK}. It turns out that
\begin{alignat*}{2}
[J_p, H(1,-3),H(1,-1,-1,-1),H(-1,1,1,1)]\equiv& [1023, 529, 670, 952]  &\pmod{1093},\\
[J_p, H(1,-3),H(1,-1,-1,-1),H(-1,1,1,1)]\equiv& [1618, 2160, 1620, 540] &\pmod{3511}.
\end{alignat*}
In order to understand the general mod $p$ structure of AMHS
we need to consider some infinite algebras similar to the adeles
(see \cite{pmod}).

\section{Congruence modulo prime powers}\label{sec:highPower}
In this last section we shall study the congruence properties of AMHS of small weights
modulo higher powers of primes $p$. We first need some results concerning the
sums $U(\ors;p-1)$ defined by \eqref{equ:Udefn}.

\begin{prop}  \label{prop:power2}
Let $p$ be a prime $\geq 7$. Let $A$ and $B$ be defined as in \S\ref{sec:wt4}.
Write $U(-)=U(-;p-1)$. Then we have
{\allowdisplaybreaks
\begin{align}
 &U(-1)\equiv -2q_p-\frac{7}{12}p^2B_{p-3} &\pmod{p^3}, \label{equ:U(-1)}\\
 &U(-2)\equiv -q_p^2+\frac{2}{3}p q_p^3+\frac{7}{6}p B_{p-3} &\pmod{p^2},  \label{equ:U(-2)}\\
 &U(-1,1)\equiv q_p^2-\frac{2}{3}p q_p^3-\frac{1}{12}p B_{p-3}&\pmod{p^2},  \label{equ:U(-1,1)}\\
 &U(1,-1)\equiv -\frac{13}{12}p B_{p-3} &\pmod{p^2},  \label{equ:U(1,-1)}\\
 &U(-3)\equiv -\frac{1}{3}q_p^3-\frac{7}{24} B_{p-3} &\pmod{p},  \label{equ:U(-3)}\\
 &U(-2,1)\equiv \frac{1}{3}q_p^3-\frac{23}{24} B_{p-3} &\pmod{p},  \label{equ:U(-2,1)}\\
 &U(1,-2)\equiv \frac{5}{4} B_{p-3} &\pmod{p},  \label{equ:U(1,-2)}\\
 &U(2,-1)\equiv -\frac{3}{4} B_{p-3} &\pmod{p},  \label{equ:U(2,-1)}\\
 &U(-1,2)\equiv \frac{1}{3}q_p^3+\frac{25}{24} B_{p-3} &\pmod{p},  \label{equ:U(-1,2)}\\
 &U(1,1,-1)\equiv -\frac{1}{2} B_{p-3} &\pmod{p}, \label{equ:U(1,1,-1)}\\
 &U(1,-1,1)\equiv \frac{1}{2} B_{p-3} &\pmod{p},  \label{equ:U(1,-1,1)}\\
 &U(-1,1,1)\equiv -\frac{1}{3}q_p^3-\frac{7}{24} B_{p-3} &\pmod{p},  \label{equ:U(-1,1,1)}\\
 &U(-4)\equiv  H(-1,1,1,1)\equiv -H(1,1,1,-1) &\pmod{p}, \label{equ:U(-4)}\\
 &U(1,-3)  \equiv A-B+\frac54 q_p B_{p-3}  &\pmod{p} ,  \label{equ:U(1,-3)}\\
 &U(-3,1)  \equiv H(1,1,1,-1)+B-A-\frac54 q_p B_{p-3}  &\pmod{p}.  \label{equ:U(-3,1)}
\end{align}}
\end{prop}
\begin{proof} Throughout the proof we write $H(-)=H(-;p-1)$ and similarly for $U$ and $V$.
We will prove the congruences in the following
order: \eqref{equ:U(-3)}, \eqref{equ:U(-1)}, \eqref{equ:U(-2,1)}, \eqref{equ:U(1,-2)},
\eqref{equ:U(-2)},  \eqref{equ:U(-1,1)},
\eqref{equ:U(1,-1)}, \eqref{equ:U(1,1,-1)}, \eqref{equ:U(2,-1)}, \eqref{equ:U(-1,2)},
and \eqref{equ:U(1,-1,1)} to \eqref{equ:U(-3,1)}.

First, \eqref{equ:U(-3)} follows from \cite[(4)]{DS} and
$H(-3)\equiv -\frac12 B_{p-3}$ (take $a=3$ in \eqref{equ:Hdepth1}).
Taking $d=1$ and $m=p-1$ in \eqref{equ:mygeneral(1)} we have
\begin{align}
\sum_{k=1}^{p-1} \frac{(1-x)^k-1}{k}
=& \sum_{j=1}^{p-1} \frac{(-x)^j}{j} {p-1\choose j}
= \sum_{j=1}^{p-1} \frac{(-x)^j}{j}\left({p\choose j}-{p-1\choose j-1}\right) \notag \\
=& -p\sum_{j=1}^{p-1} \frac{x^j}{j^2}(-1)^{j-1}{p-1\choose j-1}
-\frac{1}{p}\sum_{j=1}^{p-1} (-x)^j {p\choose j} \notag \\
=& -p\sum_{j=1}^{p-1} \frac{x^j}{j^2}(-1)^{j-1}{p-1\choose j-1}
+\frac{(x-1)^p-x^p+1}{p}. \label{equ:taulastStep}
\end{align}
Since
$$(-1)^{j-1}{p-1\choose j-1}\equiv 1-pH(1;j-1)\pmod{p^2}$$
letting $x=-1$ in \eqref{equ:taulastStep} we get
\begin{equation}\label{equ:taux=-1}
U(-1)-H(1)\equiv  -pH(-2)+p^2H(1,-2)-2q_p \pmod{p^3}.
\end{equation}
So it is not hard to see that \eqref{equ:U(-1)} can be obtained from the following:
$H(1,-2)\equiv \frac14 B_{p-3}\pmod{p}$
by \cite[Cororllary~2.4]{Tau1}, $H(-2)\equiv \frac12 p B_{p-3} \pmod{p^2}$ by \eqref{equ:Hdepth1},
and the well-known fact $H(1)\equiv -\frac13 p^2B_{p-3}\pmod{p^3}$ (see, for e.g., \eqref{equ:SunThm5.1}).
Letting $x=1/2$ in \eqref{equ:taulastStep} we have that
\begin{equation*}
 V(-1)-H(1)\equiv -pV(-2)+p^2V(1,-2)+\frac{q_p}{2^{p-1}} \pmod{p^3}.
\end{equation*}
Multiplied by $-2^p$ this yields by Lemma~\ref{lem:Ureverse}
\begin{equation}\label{equ:taux=1/2}
 U(-1)+2^p H(1)\equiv p^2U(-3)+p^2U(-2,1)-2q_p \pmod{p^3}.
\end{equation}
Thus \eqref{equ:U(-2,1)} follows from \eqref{equ:U(-1)}, \eqref{equ:U(-3)} and
$H(1)\equiv -\frac13 p^2B_{p-3}\pmod{p^3}.$ Further, by the stuffle relation
$$U(1,-2)\equiv H(1)U(-2)-U(-2,1)-U(-3)\equiv \frac{5}{4} B_{p-3} \pmod{p}$$
we get \eqref{equ:U(1,-2)}.
Now letting $x=2$ in \eqref{equ:taulastStep} we see that
\begin{equation}\label{equ:taux=2}
 H(-1)-H(1)\equiv -pU(-2)+p^2U(1,-2)-2q_p \pmod{p^3}.
\end{equation}
Hence \eqref{equ:U(-2)} follows from \eqref{equ:U(1,-2)} and \eqref{equ:H-1}.
Then taking $d=2$ and $x=-1$ in \eqref{equ:mygeneral(2)} we get
\begin{align}
U(-1,1)+U(-2)\equiv& H(1,1)+H(2)+ H(-2)-pH(1,-2)- pH(-3)\notag\\
\equiv&  \frac{13}{12} p B_{p-3}  \pmod{p^2}. \label{equ:myIIx=-1}
\end{align}
Thus $U(1,-1)=U(-1)H(1)-U(-1,1)-U(-2)\equiv -\frac{13}{12} p B_{p-3}  \pmod{p^2}$ which is
\eqref{equ:U(-1,1)}. Then \eqref{equ:U(1,-1)} follows from the stuffle relation
$U(-1,1)=H(1)U(-1)-U(1,-1)-U(-2)$.
Moreover, taking $d=1$ and $x=2$ in \eqref{equ:mygeneral(2)} we get
\begin{equation*}
     H(-1)-H(1)\equiv U(-1)-p\Big(U(-2)+U(1,-1)\Big)+p^2\Big(U(1,-2)+U(1,1,-1)\Big). \pmod{p^3}
\end{equation*}
This implies \eqref{equ:U(1,1,-1)} because of
\eqref{equ:U(-1)}, \eqref{equ:U(-2)},  \eqref{equ:U(-1,1)} and \eqref{equ:U(1,-2)}.

Next, taking $d=3$ and $x=-1$ in \eqref{equ:mygeneral(2)} we get
\begin{align}
U(-1,1,1)+U(-1,2)+U(-2,1)\equiv & H(-3)+S(1,1,1)-U(-3)  \notag\\
\equiv &\frac{1}{3}q_p^3-\frac{5}{24} B_{p-3} \pmod{p}\label{equ:myIIIx=-1}
\end{align}
since $S(1,1,1)=H(1,1,1)+H(1)H(2)\equiv 0 \pmod{p}$.
Taking $d=3$ and $x=1/2$ in \eqref{equ:mygeneral(2)} we get
\begin{equation*}
 V(-1,1,1)+V(-2,1)+V(-1,2)\equiv 0\pmod{p}
\end{equation*}
which implies by Lemma \ref{lem:Ureverse}
\begin{equation}\label{equ:myIIIx=1/2}
 U(1,1,-1)+U(1,-2)+U(2,-1)\equiv 0\pmod{p}.
\end{equation}
Thus \eqref{equ:U(2,-1)} follows from  \eqref{equ:U(1,-2)} and  \eqref{equ:U(1,1,-1)} immediately.
Then \eqref{equ:U(-1,2)} follows easily from the stuffle relation of $H(2)U(-1)$.
Also, \eqref{equ:U(1,-1,1)} and \eqref{equ:U(-1,1,1)} follow from the stuffle relations:
\begin{align*}
 H(1)U(1,-1)=&U(2,-1)+U(1,-2)+ 2U(1,1,-1)+ U(1,-1,1) \equiv 0 \pmod{p},\\
 H(1)U(-1,1)=&U(-2,1)+U(-1,2)+ 2U(-1,1,1)+ U(1,-1,1) \equiv 0 \pmod{p}.
\end{align*}
We now turn to the last three congruences of weight four.
By \cite[Theorem~2.3]{Tau1} we know \eqref{equ:U(-4)} holds.
Hence taking $d=4$, $m=p-1$ and $x=2$ in Lemma~\ref{lem:binomialU}
and using \eqref{equ:p-1choosej} we get
\begin{multline}\label{equ:tauS(-1,1,1)}
S(-1,1,1)+ p\Big(U(-4)+H(-1,2,1)+H(-2,1,1)+H(-3,1)-H(1)S(-1,1,1)\Big)\\
\equiv U(-3)-H(1)\pmod{p^2}.
\end{multline}
On the other hand, taking $d=2$ and $x=2$ in \eqref{equ:mygeneral(2)} we get
\begin{equation}\label{equ:myS(-1,1,1)}
S(-1,1,1)-S(1,1,1)\equiv U(-3)-pU(-4)-pU(1,-3) \pmod{p^2}.
\end{equation}
Comparing \eqref{equ:tauS(-1,1,1)} and \eqref{equ:myS(-1,1,1)}, using
$S(1,1,1)=H(1,1,1)+H(1)H(2)\equiv 0 \pmod{p^2}$ by \cite[Theorem~2.13]{1stpart}, and
$H(-1,2,1)+H(-2,1,1)\equiv\frac54 q_p B_{p-3} \pmod{p}$ by Proposition~\ref{prop:allwt4}, we
can deduce \eqref{equ:U(1,-3)}.
Finally, \eqref{equ:U(-3,1)} follows from \eqref{equ:U(-4)} and the stuffle relation
$U(-3,1)=H(1)U(-3)-U(1,-3)-U(-4)$.
This finishes the proof of the proposition.
\end{proof}

\begin{lem} \label{lem:reversalModpsq}
Let $a,b\in \ZZ^*$ and write $H(-)=H(-;p-1)$.  Then
\begin{equation*}
H(a,b)\equiv (-1)^{a+b}\sign(ab)\Big(H(b,a) + p|b| H(\sign(b)+b,a)+p |a| H(b,\sign(a)+a) \Big)
 \pmod{p^2}.
\end{equation*}
\end{lem}
\begin{proof} By definition
\begin{align*}
H(a,b)=&\sum_{1\le m<n< p}  \frac{\sign(a)^m \sign(b)^n}{m^{|a|} n^{|b|} }
    =\sum_{1\le n<m< p}  \frac{\sign(a)^{p-m} \sign(b)^{p-n} }{(p-m)^{|a|} (p-n)^{|b|} } \\
=&(-1)^{a+b} \sign(ab) \sum_{1\le n<m<p}
    \frac{\sign(b)^n \sign(a)^m }{ n^{|b|} m^{|a|}} \frac{1} { (1-p/n)^{|b|} (1-p/m)^{|a|} }.
\end{align*}
The lemma follows easily.
\end{proof}

\begin{prop} \label{prop:Hdepth2}
Let $A$ and $B$ be defined as in \S\ref{sec:wt4}. For all prime $p\ge 7$
write $H(-)=H(-;p-1)$ and set (see Theorem~\ref{thm:sun}(b))
$$\XX=\XX_p(3):=\frac{B_{p-3}}{p-3}
    - \frac{B_{2p-4}}{4p-8}.$$ Then we have
\begin{alignat}{2}
\label{equ:H(-1,-1)}
H(-1,-1)\equiv& 2q_p^2+p\left(2\XX-2q_p^3\right)
+p^2\left(\frac{11}{6}q_p^4+\frac12 q_pB_{p-3} \right) &\pmod{p^3}\, \\
    \equiv& 2q_p^2-2pq_p^3-\frac13 pB_{p-3} &\pmod{p^2}, \label{equ:H(-1,-1)(2)}\\
H(1,-1)\equiv & q_p^2 - pq_p^3 - \frac{13}{24}p B_{p-3} &\pmod{p^2}, \label{equ:H(1,-1)}\\
H(-1,1)\equiv & -q_p^2 +pq_p^3 +\frac{1}{24}p B_{p-3} &\pmod{p^2}, \label{equ:H(-1,1)}\\
H(-3) \equiv &  3\XX &\pmod{p^2}, \label{equ:H(-3)}\\
H(-2,1)\equiv& H(1,-2)\equiv
-\frac32 \XX &\pmod{p^2}, \label{equ:H(-2,1)}\\
H(2,-1)\equiv&  -\frac32 \XX-\frac76p q_p B_{p-3}+p(B-A) &\pmod{p^2}, \label{equ:H(2,-1)}\\
H(-1,2)\equiv& -\frac32 \XX-\frac16p q_p B_{p-3}+p(A-B) &\pmod{p^2}. \label{equ:H(1,-2)}
\end{alignat}
\end{prop}
\begin{proof} First, \eqref{equ:H(-1,-1)} follows readily from
the shuffle relation $H(-1)^2=2H(-1,-1)+H(2)$, the congruence \eqref{equ:H-1}, and
$H(2)\equiv -4p \XX \pmod{p^3}$ by \eqref{equ:SunThm5.1}.
This clearly implies \eqref{equ:H(-1,-1)(2)} since by Kummer congruence
$X\equiv -B_{p-3}/6 \pmod{p}$.

Next, taking $d=1$, and $x=-1$ in \eqref{equ:mygeneral(2)} we get
\begin{equation*}
     U(-1)-H(1)\equiv H(-1)-p\Big(H(-2)+H(1,-1)\Big) \pmod{p^2}.
\end{equation*}
Combining this with \eqref{equ:U(-1)} and using \eqref{equ:H-1} we get
\begin{equation*}
  H(1,-1)\equiv q_p^2 - \frac 23 pq_p^3 - \frac{1}{4}p B_{p-3}+pH(1,1,-1)\equiv
  q_p^2 - pq_p^3 - \frac{13}{24}p B_{p-3} \pmod{p^2}
\end{equation*}
which is \eqref{equ:H(1,-1)}.  Then \eqref{equ:H(-1,1)} is deduced from the stuffle relation
\begin{equation*}
  H(-1,1)\equiv H(-1)H(1)- H(1,-1)-H(-2) \equiv
  -q_p^2 +pq_p^3 +\frac{1}{24}p B_{p-3} \pmod{p^2}.
\end{equation*}

Turning to weight three we get by \cite[Theorem~2.1]{Tau1}
\begin{equation*}
     H(-3)\equiv  \frac14 H(3;(p-1)/2) \pmod{p^2}.
\end{equation*}
Hence \eqref{equ:H(-3)} follows from \eqref{equ:SunThm5.2}.
Now by Lemma \ref{lem:reversalModpsq} we have the reversal relation
\begin{equation}\label{equ:H(-2,1)rev}
H(-2,1)\equiv H(1,-2)+p H(2,-2)+2p H(1,-3)\equiv H(1,-2) \pmod{p^2}
\end{equation}
by Proposition~\ref{prop:wt4depth2}. On the other hand, by stuffle relation
\begin{equation*}
H(-2,1)+H(1,-2)=H(1)H(-2)-H(-3)\equiv -3\XX\pmod{p^2}.
\end{equation*}
Together with \eqref{equ:H(-2,1)rev} this clearly yields
\eqref{equ:H(-2,1)}.

Finally, by stuffle relation and \eqref{equ:H(-3)}
\begin{equation}\label{equ:H(2-1)stuff}
H(2,-1)+H(-1,2)=H(2)H(-1)-H(-3)
 \equiv -\frac43p q_p B_{p-3} -3\XX \pmod{p^2}.
\end{equation}
By Lemma \ref{lem:reversalModpsq} we have the reversal relation
\begin{equation*}
H(2,-1)\equiv H(-1,2)+p H(-2,2)+2p H(-1,3)\equiv H(1,-2)+2p(B-A)-p q_p B_{p-3} \pmod{p^2}
\end{equation*}
by Proposition~\ref{prop:wt4depth2}.
Combining with \eqref{equ:H(2-1)stuff} this implies
\eqref{equ:H(2,-1)} and \eqref{equ:H(1,-2)}.

We have finished the proof of the proposition.
\end{proof}

\begin{cor} For every prime $p\ge 7$ we have
\begin{equation} \label{equ:U(-1)2}
U(-1;p-1) \equiv -2q_p+ \frac72 \XX p^2+\frac12 p^3  H(-3,1;p-1)  \pmod{p^4}.
\end{equation}
\end{cor}
\begin{proof}  Write $U(-)=U(-;p-1)$ and $H(-)=H(-;p-1)$.
Taking $x=-1$ in \eqref{equ:taulastStep} we get
\begin{equation*}
U(-1)-H(1)\equiv  -pH(-2)+p^2H(1,-2)-p^3 H(1,1,-2)-2q_p \pmod{p^4}.
\end{equation*}
Thus the corollary follows from \eqref{equ:H(-2,1)}, Proposition~\ref{prop:allwt4},
the following congruences
\begin{alignat*}{2}
   H(1)\equiv&  2 p^2 \XX & \pmod{p^4} \quad \text{(by \cite[Remark~5.1]{Sun}),}\\
   H(2)\equiv& -4 p \XX   & \pmod{p^3}    \quad \text{(by \eqref{equ:SunThm5.1}),\ \phantom{(Remark5.1)}}\\
H(2,(p-1)/2)\equiv&  -14 p\XX & \pmod{p^3}\quad \text{(by \eqref{equ:SunThm5.2}),\ \phantom{(Remark5.1)}}
\end{alignat*}
and $H(-2)=\frac12H(2,(p-1)/2)-H(2).$
\end{proof}

\begin{prop} \label{prop:Hdepth2neg}
For all prime $p\ge 7$  write $H(-)=H(-;p-1)$, $h_{31}:= H(3,1;(p-1)/2)$, and set
$\XX$ as above. Then
\begin{align}\label{equ:H(-1,-2)}
H(-1,-2)  \equiv &\ \ \ \frac92 \XX
  -\frac56 p q_p B_{p-3} +\frac14 p  h_{31} &\pmod{p^2},\\
H(-2,-1) \equiv & -\frac92 \XX
    -\frac16p q_p B_{p-3}-\frac14 p h_{31} &\pmod{p^2}.  \label{equ:H(-2,-1)}
\end{align}
\end{prop}
\begin{proof}
By expanding $H(1,2;p-1)$ we get
\begin{alignat}{2}
& H(1,2;(p-1)/2)-H(2,1;(p-1)/2) & \ \\
\equiv&
H(1,2;p-1)-H(1;(p-1)/2) \Big( H(2;(p-1)/2)+2pH(3;(p-1)/2) \Big) \\
& \hskip 1cm
+pH(2,2;(p-1)/2)+2ph_{31} &\pmod{p^2}\,\\
\equiv& -6\XX - \frac{10}{3} p q_p  B_{p-3}
+2ph_{31}  &\pmod{p^2}
\end{alignat}
since $2H(2,2;(p-1)/2)=H(2;(p-1)/2)^2-H(4;(p-1)/2)\equiv 0\pmod{p}$
and
\begin{equation}\label{equ:H(1,2)}
H(1,2)\equiv -H(2,1)\equiv -6\XX \pmod{p^2}
\end{equation}
by \cite[Theorem 2.3]{Tau2}. Hence
\begin{alignat*}{2}
H(-1,-2)\equiv & \frac14\Big(H(1,2;(p-1)/2)-H(2,1;(p-1)/2)-p h_{31}\Big)-H(1,2) &\pmod{p^2}\, \\
 \equiv & \frac92 \XX
  -\frac56 p q_p B_{p-3} +\frac14 p h_{31} &\pmod{p^2}
\end{alignat*}
which is \eqref{equ:H(-1,-2)}. Finally, \eqref{equ:H(-2,-1)} follows from
\begin{equation*}
 H(-1,-2)+H(-2,-1)=H(-1)H(-2)-H(3) \equiv -p q_p B_{p-3} \pmod{p^2}.
\end{equation*}
This completes the proof of the proposition.
\end{proof}

Finally, we consider the depth three cases.
\begin{prop} \label{prop:Hdepth3}
Let $A$ and $B$ be defined as in \S\ref{sec:wt4}. For all prime $p\ge 7$ set
$\XX$ as above, write $H(-)=H(-;p-1)$
and $h_{31}:= H(3,1;(p-1)/2)$ as above. Then we have
\begin{alignat}{2}
H(-1,-1,-1)\equiv&  -\frac43 q_p^3+\XX+
p\Big(2q_p^4+\frac23 q_p B_{p-3}\Big) &\pmod{p^2}, \label{equ:H(-1,-1,-1)}\\
H(-1,1,-1)\equiv& \frac p2 \Big(q_p B_{p-3}+B - A\Big)  &\pmod{p^2}, \label{equ:H(-1,1,-1)} \\
H(1,-1,-1)\equiv& -q_p^3+\frac{21}4 \XX+ p
\Big(\frac32 q_p^4+\frac38 q_p B_{p-3}
+\frac A4-\frac B4+\frac18 h_{31}\Big)  &\pmod{p^2}, \label{equ:H(1,-1,-1)}\\
H(-1,-1,1)\equiv&  q_p^3-\frac{21}4 \XX+
p \Big(-\frac32 q_p^4+\frac18 q_p B_{p-3}
+\frac A4   -\frac B4-\frac18 h_{31}\Big) &\pmod{p^2}, \label{equ:H(-1,-1,1)}\\
H(1,-1,1)\equiv&-2 H(1,1,-1)+3\XX+ p\Big(\frac76 q_p B_{p-3}+A-B\Big), &\pmod{p^2},  \label{equ:H(1,-1,1)}\\
H(-1,1,1)\equiv& H(1,1,-1)-p \Big(\frac12 q_p B_{p-3}+A-B \Big)&\pmod{p^2}. \label{equ:H(-1,1,1)}
\end{alignat}
\end{prop}
\begin{proof}
First, \eqref{equ:H(-1,-1,-1)}
follows from the stuffle relation
$$3H(-1,-1,-1)=H(-1)H(-1,-1)-H(-1,2)-H(2,-1).$$
By reversal relation (similar to the proof of Lemma~\ref{lem:reversalModpsq})
we can show easily that
\begin{equation*}
 H(-1,1,-1)\equiv -H(-1,1,-1)-p\Big(H(-2,1,-1)+H(-1,2,-1)+H(-1,1,-2) \Big) \pmod{p^2}.
\end{equation*}
Notice that $H(-1,1,-2)\equiv H(-2,1,-1)\pmod{p}$. Hence
\eqref{equ:H(-1,1,-1)} follows from Proposition~\ref{prop:allwt4}.
It then implies \eqref{equ:H(1,-1,-1)} and \eqref{equ:H(-1,-1,1)}
by the two stuffle relations:
\begin{align*}
 2H(1,-1,-1)=&H(-1)H(1,-1)-H(-1,1,-1)-H(-2,-1)-H(1,2),\\
 2H(-1,-1,1)=&H(-1)H(-1,1)-H(-1,1,-1)-H(-1,-2)-H(2,1),
\end{align*}
Proposition \ref{prop:Hdepth2}, Proposition \ref{prop:Hdepth2neg}
and \eqref{equ:H(1,2)}. Similarly,
the last two congruences follow immediately from the stuffle relations:
\begin{align*}
 H(1,-1,1)=&H(1)H(1,-1)-2H(1,1,-1)-H(2,-1)-H(1,-2),\\
 2H(-1,1,1)=&H(1)H(-1,1)-H(1,-1,1)-H(-2,1)-H(-1,2),
\end{align*}
and Proposition \ref{prop:Hdepth2}. This completes the proof of the proposition.
\end{proof}
\begin{rem}
Currently, we are not able to express $H(1,1,-1)\pmod{p^2}$
explicitly. In fact, by \eqref{equ:tauS(-1,1,1)} it is not hard to find that
$$H(1,1,-1)\equiv
U(-3)-p\Big(U(-4)+\frac{7}{12}q_pB_{p-3}+A-B\Big)\pmod{p^2},$$
hence it is equivalent to determining $U(-3)\pmod{p^2}$.
\end{rem}

In conclusion, we remark that to study weight $w$ AMHS modulo $p^2$ we
need information of weight $w+1$ AMHS modulo $p$, and if we change modulus
to $p^3$ then we would need to know AMHS of weight $w+2$ modulo $p$.
Hence, it is possible that results in \cite{Agoh} might provide some
help in determining the congruence of weight three AMHS modulo prime
squares.

Email Address: tauraso@axp.mat.uniroma2.it; zhaoj@eckerd.edu

\end{document}